\documentclass[11pt]{amsart}

\usepackage{epigamath}

\usepackage[english]{babel}

\numberwithin{equation}{section}

\usepackage{enumitem}

\usepackage{amscd, amssymb, amsmath, mathrsfs}
\usepackage[english]{babel}
\usepackage{booktabs}
\usepackage{tikz-cd}
\usepackage{url}

\setlist[enumerate,1]{label={\rm(\roman*)}, ref={\rm\roman*}}

\newtheorem{theorem}{Theorem}[section]
\newtheorem*{maintheorem}{Theorem}
\newtheorem{lemma}[theorem]{Lemma}
\newtheorem{proposition}[theorem]{Proposition}
\newtheorem{corollary}[theorem]{Corollary}

\theoremstyle{definition}
\newtheorem{definition}[theorem]{Definition}

\theoremstyle{remark}

\newcommand{\NN}{\mathbb{N}}
\newcommand{\ZZ}{\mathbb{Z}}

\newcommand{\FF}{\mathbb{F}}

\newcommand{\PP}{\mathbb{P}}
\renewcommand{\AA}{\mathbb{A}}
\newcommand{\GG}{\mathbb{G}}

\newcommand{\ideala}{\mathfrak{a}}

\newcommand{\shN}{\mathscr{N}}
\newcommand{\shL}{\mathscr{L}}
\newcommand{\shP}{\mathscr{P}}

\newcommand{\shT}{\mathscr{P}}

\newcommand{\shZ}{\mathscr{Z}}

\newcommand{\catC}{\mathcal{C}}

\newcommand{\Ad}{\operatorname{Ad}}

\newcommand{\Aff}{\text{\rm Aff}}
\newcommand{\alg}{\text{\rm alg}}

\newcommand{\Aut}{\operatorname{Aut}}

\newcommand{\Bl}{\operatorname{Bl}}

\newcommand{\Card}{\operatorname{Card}}

\newcommand{\Der}{\operatorname{Der}}

\newcommand{\edim}{\operatorname{edim}}

\newcommand{\Fil}{\operatorname{Fil}}

\newcommand{\Frac}{\operatorname{Frac}}

\newcommand{\GL}{\operatorname{GL}}
\newcommand{\gl}{\operatorname{\mathfrak{gl}}}

\newcommand{\grp}{\operatorname{grp}}

\newcommand{\id}{{\operatorname{id}}}

\newcommand{\Kernel}{\operatorname{Ker}}

\newcommand{\Lie}{\operatorname{Lie}}

\newcommand{\lra}{\longrightarrow}

\newcommand{\Mat}{\operatorname{Mat}}

\newcommand{\maxid}{\mathfrak{m}}

\newcommand{\Nil}{\operatorname{Nil}}

\newcommand{\primid}{\mathfrak{p}}
\renewcommand{\O}{\mathscr{O}}

\newcommand{\op}{\text{\rm op}}

\newcommand{\PGL}{\operatorname{PGL}}

\newcommand{\pr}{\operatorname{pr}}
\newcommand{\Proj}{\operatorname{Proj}}

\newcommand{\quadand}{\quad\text{and}\quad}

\newcommand{\ra}{\rightarrow}

\newcommand{\red}{{\operatorname{red}}}

\newcommand{\Sh}{{\operatorname{Sh}}}

\newcommand{\Sing}{\operatorname{Sing}}

\newcommand{\Spec}{\operatorname{Spec}}

\newcommand{\Supp}{\operatorname{Supp}}
\newcommand{\Sym}{\operatorname{Sym}}

\newcommand{\Torsion}{\operatorname{Torsion}}
\newcommand{\trdeg}{\operatorname{trdeg}}

\newcommand{\Twist}{\operatorname{Twist}}

\newcommand{\uHom}{\underline{\operatorname{Hom}}}

\newcommand{\lieg}{\mathfrak{g}}

\newcommand{\lieu}{\mathfrak{u}}
 \newcommand{\liegl}{\mathfrak{gl}}

\newcommand{\supth}[1]{\ensuremath{#1^{\mathrm{th}}}}

\EpigaVolumeYear{7}{2023} \EpigaArticleNr{23} \ReceivedOn{April 12, 2023}
\InFinalFormOn{September 11, 2023}
\AcceptedOn{September 29, 2023}

\title{Generalizations of quasielliptic curves}
\titlemark{Generalizations of quasielliptic curves}

\author{Cesar Hilario}
\address{Mathematisches Institut, Heinrich-Heine-Universit\"at, 40204 D\"usseldorf, Germany}
\email{Cesar.Hilario@hhu.de}
\author{Stefan Schr\"oer}
\address{Mathematisches Institut, Heinrich-Heine-Universit\"at, 40204 D\"usseldorf, Germany}
\email{schroeer@math.uni-duesseldorf.de}

\authormark{C.~Hilario and S.~Schr\"oer}

\AbstractInEnglish{We generalize the notion of quasielliptic curves,
  which have infinitesimal symmetries and exist only in characteristic
  two and three, to a hierarchy of regular curves having infinitesimal
  symmetries, defined in all characteristics and having higher genera.
  This relies on the study of certain infinitesimal group schemes
  acting on the affine line and certain compactifications.  The group
  schemes are defined in terms of invertible additive polynomials over
  rings with nilpotent elements, and the compactification is
  constructed with the theory of numerical semigroups.  The existence
  of regular twisted forms relies on Brion's recent theory of
  equivariant normalization.  Furthermore, extending results of Serre
  from the realm of group cohomology, we describe non-abelian
  cohomology for semidirect products, to compute in special cases the
  collection of all twisted forms.}



\acknowledgement{The research was conducted in the framework of the
  research training group \emph{GRK 2240: Algebro-Geometric Methods in
  Algebra, Arithmetic and Topology}.}

\begin{document}



\maketitle

\begin{prelims}

\DisplayAbstractInEnglish

\bigskip

\DisplayKeyWords

\medskip

\DisplayMSCclass







\end{prelims}


\newpage

\setcounter{tocdepth}{1}

\tableofcontents


\section{Introduction}\label{Introduction}

Let $K$ be a ground field of characteristic $p>0$. The goal of this paper is to generalize, in an equivariant way, the \emph{rational cuspidal curve}
\begin{equation}
\label{rational cuspidal curve}
X=\Spec K[T^2,T^3]\cup \Spec K[T^{-1}] 
\end{equation}
from the cases $p=2$ and $p=3$, when the automorphism group scheme is non-reduced, to a hierarchy of integral curves $X_{p,n}$ whose automorphism group schemes are likewise non-reduced.  Here the index $p>0$ indicates the characteristic, and $p^{n(n+1)/2}$ gives the ``size'' of non-reducedness.

Our motivation originates from the \emph{Enriques classification} of algebraic surfaces over ground fields $k=k^\alg$: This vast body of theorems on the structure of surfaces $S$ was extended by Bombieri and Mumford to positive characteristics (see \cite{Bombieri; Mumford 1977, Bombieri; Mumford 1976}).  Their main insight was the introduction and analysis of \emph{quasielliptic fibrations}, which are morphisms $f\colon S\ra B$ whose generic fiber $Y=f^{-1}(\eta)$ is a so-called \emph{quasielliptic curve}, in other words, a twisted form of \eqref{rational cuspidal curve} over the function field $K=k(B)$, with all local rings $\O_{Y,y}$ regular.  One knows that such twisted forms exist only over imperfect fields of characteristic $p\leq 3$, and Queen gave explicit equations for them (see \cite{Queen 1971, Queen 1972}), although of rather extrinsic nature.

We discovered the hierarchy $X=X_{p,n}$ somewhat accidentally, while seeking a deeper and more intrinsic understanding of quasielliptic curves.  The curves do not reveal themselves in any direct way; one has to understand them through their \emph{automorphism group scheme} $\Aut_{X/K}$.  Its crucial part consists of certain infinitesimal group schemes $U_n$ of order $ p^{n(n+1)/2}$ acting in a canonical way on the affine line $\AA^1=\Spec K[T^{-1}]$.  The underlying scheme is $\alpha_{p^n}\times\alpha_{p^{n-1}}\times\ldots\times\alpha_p$, a singleton formed with iterated Frobenius kernels of the additive group, but endowed with a non-commutative group law. Its definition relies on the so-called \emph{additive polynomials} $\sum_{i=0}^n\lambda_iT^{-p^i}$, or equivalently the elements of the \emph{skew polynomial ring} $R[F;\sigma]$, formed over rings with nilpotent elements.

According to Brion's recent theory of \emph{equivariantly normal curves}, see \cite{Brion 2022b}, there is a unique compactification $\AA^1\subset X_{p,n}$ to which the action of $U_n$ extends in an optimal way.  In general, it is very difficult to unravel the structure of such compactifications, but here we were able to ``guess'' an explicit description in terms of the \emph{numerical semigroups}
$$
\Gamma_{p,n}=\left\langle p^n,p^n-p^{n-1},\ldots,p^n-p^0\right\rangle\subset\NN,
$$
monoids that each comprise all but finitely many natural numbers.  The guesswork was assisted by computer algebra computations with Magma and GAP, performed in a handful of special cases.  Our first main result is that the ensuing toric compactification has an intrinsic meaning. 

\begin{maintheorem}[See Theorems~\ref{action extends} and~\ref{equivariantly normal}] The $U_n$-action on the affine line extends to the compactification
$$
X_{p,n}=\Spec K\left[T^{\Gamma_{p,n}}\right]\cup\Spec K\left[T^{-1}\right],
$$
and this projective curve is equivariantly normal with respect to the $U_n$-action.
\end{maintheorem}

For $3\leq p^n\leq 4$, this is precisely the rational cuspidal curve.  The second main result unravels the numerical invariants and infinitesimal symmetries of this hierarchy of projective curves. 

\begin{maintheorem}[See Section~\ref{Complete intersection} and Theorem~\ref{automorphism group scheme}]
The curves $X=X_{p,n}$ have
$$
h^1(\O_X)= \tfrac{1}{2}(np^{n+1}-(n+2)p^n+2)\quadand \Aut_{X/K}=\GG_a\rtimes U_n\rtimes \GG_m.
$$
\end{maintheorem}

In this \emph{iterated semidirect product}, the additive group $\GG_a$ is normalized by the infinitesimal group scheme $U_n$, and both are normalized by the multiplicative group $\GG_m$.  Note that for $3\leq p^n\leq 4$, this precisely gives back the computation of Bombieri and Mumford \cite[Proposition~6]{Bombieri; Mumford 1976}, and the above should be seen as a natural generalization.

The computation of the genus relies on general results of Delorme \cite{Delorme 1976} on numerical semigroups, applied to our $\Gamma_{p,n}$. The determination of the automorphism group is based on further surprising properties of the projective curves $X=X_{p,n}$: The tangent sheaf $\Theta_{X/K}=\uHom(\Omega^1_{X/K},\O_X)$ turns out to be invertible, actually very ample, giving a canonical inclusion $X\subset\PP(\lieg)= \PP^{n+1}$, where $\lieg=H^0(X,\Theta_{X/K})$ is the Lie algebra of the automorphism group scheme $G=\Aut_{X/K}$.  From the canonical linearization $\O_X(1)=\Theta_{X/K}$, we get a matrix representation for $G$, which is crucial to gain control over its structure.  Furthermore, $X$ is \emph{globally a complete intersection}, defined inside $\PP^{n+1}$ by the following $n$ homogeneous equations of degree $p$:
$$
U_{n-1}^p-V^{p-1}Z=0\quadand U_j^p-V^{p-1}U_{j+1}=0\quad (0\leq j\leq n-2).
$$
Also note that the curves are related by a hierarchy of blowing-ups $X_{p,n-1}=\Bl_Z(X_{p,n})$, where the center is the singular point (\textit{cf.} Lemma~\ref{blowing-up}).

Again building on Brion's theory of equivariantly normal curves, see \cite{Brion 2022b}, we show that our $X=X_{p,n}$ have, over ground fields $K$ with ``enough'' imperfection, twisted forms $Y$ where all local rings $\O_{Y,y}$ are regular (\textit{cf.} Theorem~\ref{equivariantly normal}). These have the same structural properties of $X$, except that the singularities get ``twisted away.''  In turn, the passage from the rational cuspidal curve to quasielliptic curves is generalized to our hierarchy $X=X_{p,n}$.

The above relies on rather general observations, which form the  third main result of this paper. 

\begin{maintheorem}[See Section~\ref{Equivariant normality}]
Let $X$ be a geometrically integral curve with the action of a finite group scheme $G$.  Suppose that\, $\Sing(X/K)_\red$ is \'etale.  Then $X$ is equivariantly normal if and only if for some field extension $K\subset L$, the base-change $X\otimes L$ admits a twisted form that is regular.
\end{maintheorem}

Quasielliptic fibrations play a crucial role in the arithmetic of algebraic surfaces of special type, in particular for K3 surfaces and Enriques surfaces (for an example, see \cite{Schroeer 2023a}). We expect that twists over function fields of our $X=X_{p,n}$ play a similar role for surfaces of general type.

By the general theory of non-abelian cohomology and twisted forms, one may view the collection $\operatorname{Twist}(X) $ of isomorphism classes of twisted forms over $S=\Spec (K)$ as non-abelian cohomology $H^1(S,\Aut_{X/S})$.  For our curves $X=X_{p,n}$, we determined the automorphism group scheme.  This raises the question of how to compute non-abelian cohomology for semidirect products in general.  We  establish effective techniques to do so and are able apply them at least in the case $n\leq 2$.  Our fourth main result is as follows. 

\begin{maintheorem}[See Theorem~\ref{cohomology set}] For $G=\GG_a\rtimes U_2\rtimes \GG_m$, the non-abelian cohomology is
$$
H^1(S,G)= \bigcup K/\left\{u^{p^2}- v - \alpha v^p - \beta^p v^{p^2}\mid u,v\in K\right\},
$$
where the union runs over $(\alpha,\beta)\in \bigcup_{K/K^{p^2}} K/K^p$, with $\alpha\in K/K^{p^2}$ and $\beta\in K/K^p$.
\end{maintheorem}

Perhaps this is the first explicit determination of $\Twist(X)$ via a purely non-abelian cohomological computation of $H^1(S,\Aut_{X/S})$ for some relevant hierarchy of schemes $X$. A crucial step in this is the determination of particular twisted forms ${}^P\!\GG_a$ in the form given by Russell \cite{Russell 1970}, a technique likely to be of independent interest.

Let us quote Bombieri and Mumford \cite[p.~198]{Bombieri; Mumford 1976}: ``The study of special low characteristics can be one of two types: amusing or tedious.  It all depends on whether the peculiarities encountered are felt to be meaningful variations of the general picture [\ldots]  or are felt instead to be accidental and random, due for instance to numerological interactions [\ldots].''  We think that our results amply show that what Bombieri and Mumford have uncovered for $p\leq 3$ is indeed far from accidental, and belong to a structural hierarchy that indeed can be understood from general principles.

\medskip
The paper is organized as follows: In Section~\ref{Invertible additive} we develop the theory of additive polynomials over rings that contain nilpotents, study the resulting groups of units, and introduce $U_n(R)$.  The ensuing actions on polynomial rings are discussed in Section~\ref{Actions on polynomial}.  Building on these preparations, we give in Section~\ref{Scheme-theoretic reinterpretation} a scheme-theoretic reinterpretation and determine the Lie algebra and the upper and lower central series for the infinitesimal group scheme $U_n$.  In Section~\ref{Compactifications} we examine the equivariant compactifications of the affine line $\AA^1$ and introduce our numerical semigroup $\Gamma_{p,n}$ and the ensuing curve $X_{p,n}$, which turns out to be equivariantly normal. We determine the numerical invariants and deduce several crucial geometric consequences in Section~\ref{Complete intersection}.  In Section~\ref{Projective model} we show that our curve can also be seen as a global complete intersection $X_{p,n}\subset\PP^{n+1}$.  In Section~\ref{Automorphism group} its automorphism group scheme is determined.  Section~\ref{Equivariant normality} contains general results on the relation between equivariant normality and the existence of twisted forms that are regular, which is then applied to our curves $X_{p,n}$.  Section~\ref{Non-abelian cohomology} is devoted to twisting and the computation of non-abelian cohomology for semidirect products. We apply this in Section~\ref{Set torsors} to describe the collection of all twisted forms of $X_{p,n}$ in the cases $n=1$ and $n=2$.

\subsection*{Acknowledgments}
We heartily thank the two referees for their thorough reading and many valuable suggestions, which helped to improve the paper.

\section{Invertible additive polynomials}
\label{Invertible additive}

In this section we gather purely algebraic facts that go into the definition of our infinitesimal group scheme $U=U_n$ in Section~\ref{Scheme-theoretic reinterpretation}.  Fix some ring $R$ of characteristic $p>0$, and let $x$ be an indeterminate.  Recall that polynomials of the form
$$
P(x) = \sum_{i=0}^n \lambda_i x^{p^i}=\lambda_0x+\lambda_1x^p+\ldots+\lambda_nx^{p^n}\in R[x]
$$
are called \emph{additive polynomials}. Another widespread designation is \emph{$p$-polynomials}.  Clearly, the set of all such polynomials is stable under addition $P(x) + Q(x)$ and substitution $P(Q(x))$.  These two composition laws enjoy the distributive property. In fact, the additive polynomials form an \emph{associative ring} with respect to these laws, with zero element $P(x)=0$ and unit element $P(x)=x$.  Let us call it the \emph{ring of additive polynomials}.

It can also be seen as the \emph{skew polynomial ring} $R[F;\sigma]$, where $\sigma\colon R\ra R$ designates the Frobenius map $\lambda\mapsto\lambda^p$.  Elements are polynomials in the formal symbol $F$, and multiplication is subject to the relations $F\lambda = \lambda^pF$. In other words, we have
\begin{equation}
\label{multiplication}
\sum_i \lambda_iF^i\cdot \sum_j \mu_jF^j=\sum_k \left(\sum_{i+j=k}\lambda_i\mu_j^{p^i} \right)F^k,
\end{equation}
a modification of the usual Cauchy multiplication.  The identification of the skew polynomial ring with the ring of additive polynomials is given by $\lambda\mapsto \lambda x$ and $F\mapsto x^p$, so that $\sum\lambda_iF^i$ corresponds to $\sum \lambda_ix^{p^i}$.  For psychological reasons, we strongly prefer to make computations in the skew polynomial ring. In the next section, when it comes to actions on the affine line, we shall turn back to the ring of additive polynomials.

Over ground fields, the ring of additive polynomials was introduced and studied by Ore \cite{Ore 1933}.  A discussion from the perspective of skew polynomial rings was given by Jacobson \cite[Chapter~3]{Jacobson 1943}.  More recent presentations appear in \cite[Chapter~1]{Goss 1996} and \cite[Chapter~2]{Goodearl; Warfield 2004}.  For our purposes, however, it will be crucial to allow nilpotent elements.  The following two propositions reveal that nilpotent and invertible elements in $R[F;\sigma]$ are characterized as in usual polynomial rings.

\begin{proposition}\label{nilpotent in skew}
An element $\sum_{i=0}^n \lambda_iF^i$ of the skew polynomial ring $R[F;\sigma]$ is nilpotent if and only if $\lambda_0,\ldots,\lambda_n\in \Nil(R)$.
\end{proposition}

\begin{proof}
Suppose that all coefficients are nilpotent, say $\lambda_i^d=0$. For each $r\geq 0$, we have
$$ \left(\sum_i\lambda_iF^i\right)^r = \sum_k \left(\sum_{i_1+\ldots+i_r=k} \lambda_{i_1}^{v_1}\cdots\lambda_{i_r}^{v_r}\right) F^k
$$
for certain exponents $v_1,\ldots,v_r\geq 1$ whose precise values are irrelevant in the following reasoning: If $r>(n+1)(d-1)$, each tuple $0\leq i_1,\ldots,i_r\leq n$ must contain the $d$-fold repetition of at least one value $0\leq i\leq n$. Then the product $ \lambda_{i_1}^{v_1}\cdots\lambda_{i_r}^{v_r}$ vanishes, and so does the above $r$-fold power.

Conversely, suppose that some $\lambda_s\in R$ is not nilpotent. Choose a prime $\primid\subset R$ not containing $\lambda_s$, and set $K=\kappa(\primid)$. Then the image of $\sum_{i=0}^n \lambda_iF^i$ in the skew polynomial ring $K[F;\sigma]$ is a non-zero nilpotent element. On the other hand, $K[F;\sigma]$ is a domain (this follows from \cite[Chapter~3, Section~1, bottom paragraph on p.~29]{Jacobson 1943}), giving a contradiction.
\end{proof}

This has the following important consequence. 

\begin{proposition}\label{invertible in skew}
An element $P=\sum_{i=0}^n \lambda_iF^i$ of the skew polynomial ring $R[F;\sigma]$ is invertible if and only if $\lambda_0\in R^\times$ and $\lambda_1,\ldots,\lambda_n\in \Nil(R)$.
\end{proposition}

\begin{proof}
The condition is sufficient: Set $\mu_i=-\lambda_i/\lambda_0$.  By Proposition~\ref{nilpotent in skew}, the element $Q= \sum_{i=1}^n\mu_iF^i$ is nilpotent, say $Q^r=0$. Then $1-Q$ is a unit, with inverse $\sum_{j=0}^{r-1} Q^j$. Thus $P=\lambda_0(1-Q)$ is also a unit.

Conversely, suppose  $P Q=Q P =1$. From the group law \eqref{multiplication}, one immediately infers that $\lambda_0\in R^\times$.  Seeking a contradiction, we assume that some $\lambda_s\in R$, $s\geq 1$ is not nilpotent. Choose such $1\leq s\leq n$ maximal.  As above, we find some residue field $K=\kappa(\primid)$ in which $\lambda_s$ is non-zero.  Let $d\geq 0$ be the degree of the image of $Q$.  From \eqref{multiplication} one sees that the image of $1=P Q$ has non-zero term in degree $s+d$, giving a contradiction.
\end{proof}

Given a unit of the form $P=\sum\lambda_iF^i$, the inverse $P^{-1}=\sum\mu_jF^j$ can be computed as follows: The condition $ P\cdot P^{-1}=1$ means $\lambda_0\mu_0^{p^0}=1$ and $ \sum_{i+j=k} \lambda_i\mu_j^{p^i}=0$ for $k\geq 1$, which give the recursion formula
\begin{equation}
\label{recursion for inverse}
\mu_0=\lambda_0^{-1}\quadand \mu_k = -\frac{1}{\lambda_0}\sum_{i=1}^k \lambda_i  \mu_{k-i}^{p^i} \quad (k\geq 1).
\end{equation}
The skew polynomial ring comes with an infinite-dimensional matrix representation $R[F;\sigma]\ra\Mat_\infty(R)$, already determined by the assignments
$$
\lambda\longmapsto \begin{pmatrix} \lambda\\ & \lambda^p\\&&\lambda^{p^2}\\&&&\ddots\end{pmatrix}\quadand 
F\longmapsto \begin{pmatrix}0&1\\&0&1\\&&0&1\\&&&\ddots&\ddots \end{pmatrix}.
$$
More explicitly, this homomorphism  is given by 
\begin{equation}
\label{unitriangular}
\sum_{i=0}^n\lambda_iF^i\longmapsto \left(\lambda^{p^r}_{s-r}\right)_{0\leq r\leq s<\infty} =  
\begin{pmatrix} 
\lambda_0	&\lambda_1	&\lambda_2		&\cdots\\ 
	&\lambda_0^p&\lambda_1^p	& \lambda_2^p	&\cdots\\ 
	&	& \lambda_0^{p^2}	& \lambda_1^{p^2} 	& \lambda_2^{p^2}	&\cdots\\  
	&	&		& \ddots		&\ddots 		&\ddots
\end{pmatrix}.
\end{equation}
Obviously, the map is injective and takes values in the row-finite upper triangular matrices.  Note that for each $d\geq 0$, the top left submatrix indexed by $0\leq r,s\leq d-1$ yields a subrepresentation $R[F;\sigma]\ra\Mat_d(R)$.

We now examine the unit group $R[F;\sigma]^\times$ in more detail, for the time being as an abstract group.  It comes with matrix representations $R[F;\sigma]^\times \ra \GL_d(R)$, $d\geq 0$.  Note that this factors over the group of invertible upper triangular matrices $T_d(R)\subset\GL_d(R)$.  Write $\Fil^d$ for the kernels.  Clearly $\Fil^0=R[F;\sigma]^\times$, whereas
$$
\Fil^d=\left\{1+\sum_{i=d}^n\lambda_iF^i\mid\text{$n\geq d$ and $\lambda_i\in\Nil(R)$}\right\} \quad (d\geq 1).
$$
These form a descending chain of normal subgroups, in other words, a \emph{normal series}.  Clearly, their intersection contains only the unit element.
 
\begin{proposition}\label{normal series}
The normal series $\Fil^d$ on $R[F;\sigma]^\times$ has quotients
$$
\Fil^0/\Fil^1=R^\times\quadand \Fil^d/\Fil^{d+1}=\Nil(R) \quad(d\geq 1).
$$
Moreover, we have the commutator formula $[\Fil^1,\Fil^d]\subset\Fil^{d+1}$ for all $d\geq 1$.
\end{proposition}

\begin{proof}
  The first assertion is an immediate consequence of the group law \eqref{multiplication}.  The commutator formula follows from a corresponding commutator formula for the \emph{unitriangular group} $\operatorname{UT}_d(R)\subset \GL_d(R)$ comprising upper triangular matrices with the unit element on the diagonal (\textit{cf.} \cite[Chapter~6, Example~16.1.2]{Kargapolov; Merzljakov 1979}).
\end{proof}

The multiplicative character $R[F;\sigma]^\times \ra \GL_1(R)=R^\times$ given by $\sum\lambda_iF^i\mapsto \lambda_0$ comes with a canonical splitting $\mu\mapsto \mu F^0$, so we get a semidirect product $R[F;\sigma]^\times=\Fil^1\rtimes R^\times$.  The ensuing conjugacy action of $R^\times$ is given by
$$
\mu \left(\sum\lambda_iF^i\right)\mu^{-1} = \sum \mu^{1-p^i}\lambda_iF^i.
$$
For our applications, it will be important to consider certain smaller subgroups inside the unit group, and the following is crucial throughout.

\begin{proposition}\label{smaller subgroups}{\samepage
For each integer $n\geq 0$, the set
$$
U_n(R)   =\left\{1+\sum_{i=1}^n \lambda_iF^i \mid \text{$\lambda_i^{p^{n-i+1}}=0$ for all $1\leq i\leq n$} \right\}
$$
is a subgroup inside the unit group $R[F;\sigma]^\times$, which is normalized by $R^\times\subset R[F;\sigma]^\times$.  }
\end{proposition}

\begin{proof}
Clearly the set contains the unit element.  Suppose that $P=1+\sum_{i=1}^n\lambda_iF^i$ and $Q=1+\sum_{j=1}^n\mu_j F^j$ belong to $U_n(R)$, and write the product as $PQ=1+\sum_{k=1}^m \alpha_kF^k$, with coefficients $\alpha_k=\sum_{i+j=k}\lambda_i\mu_j^{p^i}$.  For $k\geq n+1$, each summand $\lambda_i\mu_j^{p^i}$ vanishes: If $j\geq n+1$, we already have $\mu_j=0$, and if $j\leq n$, we get $i=k-j\geq n+1-j$ and thus $\mu_j^{p^i}=0$.  For $k\leq n$, we have $\alpha_k^{p^{n-k+1}} = \sum_{i+j=k}\lambda_i^{p^{n-k+1}}\mu_j^{p^{n-j+1}}$, which vanishes because $\mu_j^{p^{n-j+1}}=0$. Thus $PQ\in U_n(R)$.

Next consider the inverse element $P^{-1}=\sum_{j\geq 0}\beta_jF^j$.  The recursion formula \eqref{recursion for inverse} gives $\beta_0=1$ and $\beta_k = - \sum_{i=1}^k \lambda_i \beta_{k-i}^{p^i} $ for $k\geq 1$.  For $k\geq n+1$, each summand $\lambda_i \beta_{k-i}^{p^i}$ vanishes because $i\geq n- (k-i)+1$.  For $k\leq n$, we have $(\lambda_i\beta_{k-i}^{p^i})^{p^{n-k+1}}=\lambda_i^{p^{n-k+1}} \beta_{k-i}^{p^{n-(k-i)+1}}$, where the second factor vanishes.  Thus $P^{-1}\in U_n(R)$.

Finally, for each $\mu\in R^\times$, we have $\mu\cdot P\cdot \mu^{-1}=1+\sum_{i=1}^n \mu^{1-p^i}\lambda_iF^i$, which clearly belongs to $U_n(R)$.  So the latter is normalized by $R^\times$.
\end{proof}

\section{Actions on polynomial rings}
\label{Actions on polynomial}

We keep the set-up as in the previous section. Obviously, the multiplicative monoid of additive polynomials $\sum \lambda_ix^{p^i}$ acts on the polynomial ring $R[x]$ via substitution of the indeterminate, in other words by $P(x)\mapsto P(\sum \lambda_ix^{p^i})$, and one easily checks that this is an action \emph{from the right}.  In turn, we have a group action $R[x]\times R[F;\sigma]^\times\ra R[x]$ from the right, given by
\begin{equation}
\label{infinitesimal action}
Q(x)\ast \sum \lambda_iF^i=Q\left(\sum \lambda_ix^{p^i}\right).
\end{equation}
Note that the action of the multiplicative group $ \GG_m(R)=R^\times$ via $Q(x)\ast\lambda_0=Q(\lambda_0x)$ is a special case of this.  Furthermore, we have the translation action of the additive group $ \GG_a(R)=R$, defined by
\begin{equation}
\label{additive action}
Q(x)\ast \alpha=Q(x+\alpha).
\end{equation}
Obviously, these actions are faithful, and we arrive at inclusions of $\GG_a(R)$ and $R[F;\sigma]^\times$ into the opposite automorphism group of $R[x]$.

\begin{proposition}
\label{iterated normalization}
Inside the opposite automorphism group of $R[x]$, the group $\GG_a(R)$ is normalized by $R[F;\sigma]^\times$, and the intersection $\GG_a(R)\cap R[F;\sigma]^\times$ is trivial.
 \end{proposition}
 
\begin{proof}
Suppose that we have elements 
$$
\alpha\in \GG_a(R)\quadand   \sum \lambda_iF^i\in R[F;\sigma]^\times.
$$
For the first assertion, it suffices to check that $P\cdot\GG_a(R) = \GG_a(R)\cdot P$.  This indeed holds because one computes $\sum\lambda_i(x+\alpha)^{p^i} = (\sum\lambda_i x^{p^i}) +\alpha'$ with $\alpha'=\sum_{i=0}^n \lambda_i\alpha^{p^i}$.  It remains to verify the assertion on the intersection.  Suppose  $\alpha=P$ as automorphisms of $R[x]$; in other words, $x+\alpha = \sum \lambda_ix^{p^i}$. Comparing coefficients at the constant terms gives $\alpha=0$; hence the intersection $\GG_a(R)\cap R[F;\sigma]^\times$ is trivial.
\end{proof}

In turn, we get an inclusion of $\GG_a(R)\rtimes R[F;\sigma]^\times$ into the opposite automorphism group of $R[x]$.  Later, we seek to extend part of this action to certain subrings of $R[x^{-1}]$ in a compatible way. The following observation will be useful: Let $S\subset R[x]$ be the multiplicative system of all monic polynomials. The resulting localization is denoted by $R(x)=S^{-1}R[x]$. Since monic polynomials are regular elements from the polynomial ring, the localization map is injective, and we get an inclusion $R[x]\subset R(x)$.

\begin{proposition}
\label{extension to fractions}
The action from the right of the group $\GG_a(R)\rtimes R[F;\sigma]^\times$ on the polynomial ring $R[x]$ uniquely extends to $R(x)$.
\end{proposition}

\begin{proof}
Uniqueness immediately follows from the universal property of localizations.  To see existence, consider the larger multiplicative system $\tilde{S}\subset R[x]$ comprising the polynomials of the form $\lambda P+Q$ with $P$ monic, $Q$ nilpotent, and $\lambda\in R^\times$.  Obviously, this system is stable with respect to the actions \eqref{infinitesimal action} and \eqref{additive action}, and we thus get an induced action on $\tilde{S}^{-1}R[x]$.  On the other hand, the inclusion $S\subset\tilde{S}$ gives a canonical map $S^{-1}R[x]\ra \tilde{S}^{-1}R[x]$.  It remains to verify that every $\lambda P+Q$ as above becomes invertible in $S^{-1}R[x]$. Indeed, in the factorization $\lambda P+Q = P/1\cdot (\lambda+ Q/ P)$, the second factor also is a unit because $\lambda$ is invertible and $Q/ P$ is nilpotent.
  \end{proof}

\section{Scheme-theoretic reinterpretation}
\label{Scheme-theoretic reinterpretation}

Fix a ground field $K$ of characteristic $p>0$. In this section we take a more geometric point of view and reinterpret and extend the results of the preceding sections in terms of schemes and group schemes.  We now regard
$$
U_n(R) =U_{n,K} (R) =  \left\{1+\sum_{i=1}^n\lambda_iF^i\in R[T;\sigma]^\times\mid \text{ $\lambda_i^{p^{n-i+1}}=0$ for $1\leq i\leq n$}\right\}
$$
as a group-valued functor $U_n$ on the category of $K$-algebras $R$. Clearly, the natural transformation
\begin{equation}
\label{underlying scheme}
\alpha_{p^n}\times\alpha_{p^{n-1}}\times\cdots\times\alpha_p\lra U_n,\quad (\lambda_1,\ldots,\lambda_n)\longmapsto 1+\sum_{i=1}^n\lambda_iF^i
\end{equation}
is an isomorphism of set-valued functors, with group laws ignored. In turn, $U_n$ is a finite group scheme with coordinate ring $\bigotimes_{i=1}^n K[x_i]\,/(x_i^{p^{n-i+1}})$ and order $|U_n|=h^0(\O_{U_n})=p^{n(n+1)/2}$.  It contains but one point and is thus an \emph{infinitesimal group scheme}.

One immediately sees that the restriction of \eqref{underlying scheme} to $\alpha_p^{\oplus n}=\alpha_p\times\ldots\times\alpha_p$ respects the group laws and gives an inclusion of group schemes $\alpha_p^{\oplus n}\subset U_n$. Furthermore, for every $m\leq n$, we have canonical inclusions $U_m\subset U_n$ of group schemes.

Recall that each scheme $X$ over our ground field $K$ comes with a \emph{relative Frobenius map} $F\colon X\ra X^{(p)}$, given in functorial terms by $X(R)\stackrel{F}{\ra} X({}_FR)=X^{(p)}(R)$.
Here ${}_FR$ denotes the abelian group $R$, viewed as an $R$-algebra via the absolute Frobenius map $f\mapsto f^p$, and $X^{(p)}=X\otimes_K({}_FK)$. Note that $R={}_FR$ as an $\FF_p$-algebra.  Hence $X(R)=X({}_FR)$ and thus $X=X^{(p)}$, provided that $X$ arises as base-change from the prime field $\FF_p$.  For our group scheme $U_n$, the relative Frobenius map takes the form
$$
U_n(R)\lra U_n({}_FR)=U_n(R),\quad 1+\sum\lambda_iF^i\longmapsto 1+\sum\lambda_i^pF^i.
$$
 
\begin{proposition}
\label{image frobenius}
The image of\, $F\colon U_n\ra U_n$ is the subgroup scheme $U_{n-1}$, and its kernel is given by $\alpha_p^{\oplus n}$. In particular, we have an identification of restricted Lie algebras $\Lie(U_n)=K^n$.
\end{proposition}
 
\begin{proof}
Obviously, the Frobenius map factors over the subgroup scheme $U_{n-1}\subset U_n$. The resulting map $F\colon U_n\ra U_{n-1}$ is indeed an epimorphism because any $R$-valued point $1+\sum\mu_iF^i$ of $U_{n-1}$ arises from the $R'$-valued point $1+\sum\lambda_i$ of $U_n$, for the fppf extension $R'=\bigotimes R[\lambda_i]\,/(\lambda_i^p-\mu_i)$.

An $R$-valued point $1+\sum\lambda_iF^i$ belongs to the kernel of the Frobenius map if and only
if $\lambda_i^p=0$, and hence  $U_n[F]=\alpha_p^{\oplus n}$.
The last assertion follows because $\Lie(\alpha_p^{\oplus n})=K^n$, and  for any group scheme, the inclusion of the Frobenius kernel
induces a bijection on Lie algebras.
\end{proof}

In turn, the relative Frobenius map $F\colon U_n\ra U_n$ yields an extension
\begin{equation}
\label{extension frobenius}
0\lra \alpha_p^{\oplus n} \lra U_n\lra U_{n-1}\lra 1.
\end{equation}
By induction on $n\geq 0$, we infer that the finite group scheme $U_n$ admits a composition series with quotients isomorphic to $\alpha_p$. In particular, $U_n$ is \emph{unipotent}.  Since all Lie brackets are trivial, the adjoint representation $\operatorname{ad}\colon \lieu_n\ra\mathfrak{gl}(\lieu_n)$ of the Lie algebra $\lieu_n=\Lie(U_n)$ is trivial, and the adjoint representation $\Ad\colon U_n\ra\GL_{\lieu_n/k}$ of the group scheme factors over the quotient $U_{n-1}$.  It is not difficult to determine the latter representation: Since the group $U_n(K)$ is trivial, we have
$$ 
\Lie(U_n)=U_n(K[\epsilon])=\left\{1+\epsilon\sum_{r=1}^n\alpha_rF^r\mid\text{$\alpha_r\in K$}\right\},
$$
where $\epsilon$ denotes an indeterminate subject to $\epsilon^2=0$.  The elements $1+\epsilon F^r$, $1\leq r\leq n$, form a basis of this $K$-vector space. With $P =\sum \lambda_iF^i$, where $\lambda_0=1$, and using the relations $\epsilon^2=0$ and $F\epsilon =0$, we get
$$
P^{-1}\cdot (1+\epsilon F^s)\cdot P  = 1+ \epsilon F^sP=
1+  \epsilon\sum_{i=0}^{n-s} \lambda_i^{p^s}  F^{s+i}= \prod_{i=0}^{n-s}\left(1+\epsilon\lambda_i^{p^s}  F^{s+i}\right).
$$
Consequently, $\Ad(P^{-1})$ sends the basis vector $e_s=1+\epsilon F^s$ to the linear combination $\sum_{r=s}^n \lambda^{p^s}_{r-s} e_r$.  Summing up, in the restricted Lie algebra $\Lie(U_n)=K^n$ all brackets and $p$-powers are zero, and the adjoint representation of the group scheme is given by $(\sum\lambda_iF^i)^{-1}\mapsto (\lambda^{p^s}_{r-s})_{n\geq r\geq s\geq 1} $.

As described in Section~\ref{Actions on polynomial}, the groups $U_n(R)$ act from the right on the polynomial ring $R[x]$ via $R$-linear maps.  This is obviously functorial in $R$ and thus constitutes an action of the group scheme $U_n$ on the affine line $\AA^1=\Spec K[x]$.  Note that this is indeed an action \emph{from the left}.  On $R$-valued points, it is given by
$$
(\lambda_1,\ldots,\lambda_n)\ast \mu=\sum_{i=0}^n\lambda_iF^i\ast \mu=\sum_{i=0}^n \lambda_i\mu^{p^i},
$$
where we set $\lambda_0=1$ for convenience. Of course, we also have the canonical actions of the multiplicative group $\GG_m$ and the additive group $\GG_a$, given via $\lambda_0\ast\mu=\lambda_0\mu$ and $\alpha\ast\mu= \mu+\alpha$, respectively.  The following generalizes a key observation of Bombieri and Mumford \cite[Proposition~6]{Bombieri; Mumford 1976}. 

\begin{proposition}
\label{normalization and intersection}
The above actions of the three group schemes on the affine line are faithful.  Inside the sheaf $\Aut_{\AA^1/K}$, the group scheme $\GG_a$ is normalized by $U_n$, and both $\GG_a$ and $U_n$ are normalized by $\GG_m$.  Moreover, the intersections
$$
\GG_a\cap U_n\quadand (\GG_a\rtimes U_n)\cap \GG_m
$$
inside the sheaf $\Aut_{\AA^1/K}$ are trivial.
\end{proposition}
 
\begin{proof}
The assertions follow from    Propositions~\ref{iterated normalization} and~\ref{smaller subgroups}.
\end{proof}

We thus have an \emph{iterated semidirect product}, for simplicity written as
\begin{equation}
\label{iterated semidirect product}
\GG_a \rtimes  U_n  \rtimes \GG_m = (\GG_a\rtimes  U_n) \rtimes \GG_m = \GG_a\rtimes  (U_n \rtimes \GG_m),
\end{equation}
acting faithfully on the affine line $\AA^1=\Spec K[x]$. In turn, we get an inclusion of restricted Lie algebras
$$
K\rtimes \Lie(U_n)\rtimes\liegl_1(K)\subset \Lie(\Aut_{\AA^1/K}) = \Der_K(K[x]).
$$ 
The elements on the left-hand side can be seen as tuples $(\alpha,\lambda_1,\ldots,\lambda_n,\lambda_0)$ and correspond to the $K$-derivation $\alpha\frac{\partial}{\partial x} + \sum_{i=1}^n\lambda_ix^{p^i} \frac{\partial}{\partial x} + \lambda_0 x\frac{\partial}{\partial x}$ of the polynomial ring $K[x]$. For example, the derivation $\lambda_1x^p\partial/\partial x\in\Lie(U_n)$ acts via $x\mapsto x+ \epsilon\lambda_1x^p$, which coincides with the action of the group element $1+\epsilon\lambda_1F\in U_n(k[\epsilon])$.

The spectrum of the function field $K(x)$  comes with  a monomorphism
\begin{equation}
\label{inclusion generic point}
\Spec K(x)\lra \Spec K[x]=\AA^1.
\end{equation}
According to Proposition~\ref{extension to fractions}, there is a unique action on $\Spec K(x)$ that makes the above morphism equivariant.

Let us close this section with some observations on central series.  Recall that for a group $G$, the \emph{lower central series} $\Gamma^r=\Gamma^rG$ and the \emph{upper central series} $Z_s=Z_sG$ are inductively defined by
$$
\Gamma^0=G,\quad  \Gamma^{r+1} = [G,\Gamma^r] \quadand Z_0=\{e\},\quad Z_{s+1}/Z_s = Z(G/Z_s).
$$
The group is \emph{nilpotent} if $\Gamma^r=\{e\}$ for some $r\geq 0$, or equivalently $Z_s=G$ for some $s\geq 0$.  Then the smallest such integers coincide, and this number $n$ is called the \emph{nilpotency class} of the group.  Note that $\Gamma_{n-r}\subset Z_r$, but usually this inclusion is not an equality.  We refer to \cite[Chapter~10]{Hall 1959} or \cite[Chapter~6]{Kargapolov; Merzljakov 1979} for basic facts on nilpotent groups.

For group schemes $G$ of finite type, one has basically the same construction, with sheafification involved.  This is straightforward for the higher centers: An $x\in G(R)$ belongs to $Z_{s+1}(R)$ if and only if it commutes with all members of $G(R')$ up to elements of $Z_s(R')$, for all flat extensions $R\subset R'$.  The situation is more complicated for the $\Gamma^r$ because their formation involves schematic images and group scheme closure with respect to the commutator maps $G\times\Gamma^r\ra\Gamma^{r+1}$; see \cite[Expos\'e $\text{VI}_B$, Section 8]{SGA 3a}.

Let us  unravel this for our   $G=U_n$: consider the  closed subschemes
\begin{equation}
\label{central series}
\{1\}=G_0\subset G_1\subset\ldots\subset G_n=U_n
\end{equation}
defined by $G_r(R)=\{1+\sum_{i=n-r+1}^n\lambda_iF^i\}$. 

\begin{proposition}
\label{lower upper series}
The $G_r\subset U_n$ are subgroup schemes, and the series \eqref{central series} coincides with both the upper and the lower central series for the group scheme $U_n$.  The quotients are $G_{r+1}/G_r=\alpha_{p^{r+1}}$.
\end{proposition}

\begin{proof}
With descending induction one easily checks that $G_r\subset G_{r+1}$ are subgroup schemes: the surjection $G_{r+1}\ra\alpha_{p^{r+1}}$ given by $1+\sum_{i=n-r}^n\lambda_iF^i\mapsto\lambda_{n-r}$ respects the group law and has kernel $G_r$. The isomorphism theorem gives the statement on the quotients.

The arguments for the higher centers rely on the following observation: The recursion formula \eqref{recursion for inverse} for inverses $\sum\gamma_iF^i = (\sum\beta_iF^i)^{-1}$ shows that each coefficient $\gamma_i=\gamma_i(\beta_0,\ldots,\beta_n)$ actually depends only on $\beta_0,\ldots,\beta_i$.  From this one easily infers
\begin{equation}
\label{characterization equivalence}
\left(\sum\alpha_iF^i\right)\cdot \left(\sum\beta_iF^i\right)^{-1}\in G_r(R)\quad\Longleftrightarrow\quad \text{$\alpha_i=\beta_i$ for $0\leq i\leq n-r$}.
\end{equation}
Write $Z_r$ for the higher centers of $U_n$. We show $Z_r\subset G_r$ by induction on $0\leq r\leq n$. The case $r=0$ is trivial.
Now suppose $r\geq 1$ and that the inclusion holds for $r-1$.
For each $x=\sum\lambda_iF^i$ from $U_n(R)$, we compute
$$
(1-\mu F)\cdot x=  x - \sum_{i=0}^{n-1}\mu\lambda_i^pF^{i+1}\quadand x\cdot(1-\mu F) = x- \sum_{i=0}^{n-1}\lambda_i\mu^{p^i} F^{i+1}.
$$
Suppose that $x$ belongs to $Z_r(R)$. Then for all $\mu\in\alpha_{p^n}(R')$ in some ring extension $R\subset R'$, the above two expressions coincide modulo $Z_{r-1}\subset G_{r-1}$.  From the equivalence \eqref{characterization equivalence}, we obtain $\mu\lambda_i^p=\lambda_i\mu^{p^i}$ for $0\leq i\leq n-r$.  For $R'=R[\mu]\,/(\mu^{p^n})$, we are in position to compare coefficients and infer $\lambda_1=\ldots=\lambda_{n-r}=0$, and thus $x\in G_r(R)$.

This completes our induction and establishes $Z_r\subset G_r$ for all $0\leq r\leq n$.  For the reverse inclusion, we use our embedding $U_n\subset \operatorname{UT}_{n+1}$ into the group of unitriangular matrices.  According to Lemma~\ref{higher commutators unitriangular} below, the $\supth{r}$ higher center of $\operatorname{UT}_{n+1}(R)$ is given by the matrices that are zero on the $n-r$ secondary diagonals above the main diagonal.  The intersection with $U_n(R)$ equals $G_r(R)$. Consequently, $G_r\subset Z_r$; thus the $G_r=Z_r$ form the upper central series.

The arguments for the higher commutator groups rely on some preliminary observations.  For elements of the form $b=1-\beta F^s -\gamma F^{s+1} + \cdots$ with any $s\geq 1$, the geometric series $(1-x)^{-1}=1+x+x^2+\cdots$ gives
$$
b^{-1} = \left(1-\beta F^s -\gamma F^{s+1} + \cdots\right) ^{-1}\equiv 1 + \beta F^s + \gamma F^{s+1} + \beta^{1+p^s}F^{2s},
$$
where the congruence means up to terms of order $s+2$. Note that the last summand is only relevant in the special case $s=1$.  With $a=1-\alpha F$, the above formula shows that the commutator $aba^{-1}b^{-1}$ is congruent to
 \begin{gather*}
\left(1-\beta F^s -\gamma F^{s+1}\right)^{-1}  + (1-\alpha F)\left(-\beta F^s -\gamma F^{s+1}\right)(1+\alpha F)\left(1+\beta F^s\right) \equiv \\
1+\beta F^s +\gamma F^{s+1} + \beta^{1+p^s}F^{2s}-\beta F^s +\alpha\beta^pF^{s+1} -\beta\alpha^{p^s}F^{s+1} - \beta^{1+p^s}F^{2s} - \gamma F^{s+1}. 
\end{gather*}
Most summands cancel, and the upshot is the commutator formula 
\begin{equation}
\label{commutator formula}
aba^{-1}b^{-1} = 1+\left(\alpha\beta^p-\beta\alpha^{p^s}\right)F^{s+1} + \cdots.
\end{equation}
Write $\Gamma^r$ for the higher commutator subgroup schemes.  According to general properties of nilpotent groups (\textit{cf.} \cite[p.~107]{Kargapolov; Merzljakov 1979}), we have $\Gamma^r\subset Z_{n-r}=G_{n-r}$.  We claim that the canonical projection
\begin{equation}
\label{commutator projection}
\Gamma^r=\left[U_n,\Gamma^{r-1}\right]   \lra G_{n-r}/G_{n-r-1} = \alpha_{p^{n-r}} 
\end{equation}
is an epimorphism.  We check this by induction on $r\geq 0$. The case $r=0$ is trivial. Suppose $r\geq 1$ and that the assertion is true for $r-1$.  According to \cite[Section~IV.2, Proposition~1.1]{Demazure; Gabriel 1970}, the iterated Frobenius kernels are the only subgroup schemes of $\alpha_{p^{n-r}}\subset\GG_a$.  Seeking a contradiction, we assume that the above map factors over $\alpha_{p^{n-r-1}}$.  Consider the ring $R=K[\alpha,\beta]\,/(\alpha^{p^n},\beta^{p^{n-r+1}})$. By our induction hypothesis, there are some faithfully flat extension $R\subset R'$ and some $R'$-valued point of the form $b=1-\beta F^{r-1} -\gamma F^r + \cdots$ from $\Gamma^{r-1}$.  With $a=1-\alpha F$, the commutator formula \eqref{commutator formula} shows that $aba^{-1}b^{-1}$ projects to $\lambda=\alpha\beta^p-\beta\alpha^{p^{r}}$ under \eqref{commutator projection}. So $\lambda^{p^{n-r-1}}=\alpha^{p^{n-r-1}}\beta^{p^{n-r}}-\beta^{p^{n-r-1}}\alpha^{p^{n-1}} $ vanishes in the ring $R'$. On the other hand, both of the appearing monomials belong to the monomial basis for $R$, giving a contradiction.  Thus \eqref{commutator projection} is an epimorphism.

We are now ready to prove that the inclusion $\Gamma^s\subset G_{n-s}$ is an equality.  Fix some $x\in G_{n-s}(R)$, and write it as $x=1+\sum_{i=s+1}^n\lambda_iF^i$.  We check that $x\in \Gamma^s(R)$ by descending induction on $s\leq n$.  The case $s=n$ is trivial. Now assume $s<n$ and that the assertion holds for $s+1$.  By the preceding paragraph, there are some faithfully flat extension $R\subset R'$ and some $R'$-valued point $y=1+\sum_{i=s+1}^n\mu_iF^i$ from $\Gamma^s$ with $\mu_{s+1}=-\lambda_{s+1}$.  Then $xy=1+\sum_{i=s+2}\lambda'_iF^i$ belongs to $G_{n-s-1}(R')$. Using our induction hypothesis, together with the inclusion $\Gamma^{s+1}\subset\Gamma^s$, we see that $xy$, and hence $x$,
belongs to $\Gamma^s(R')$, and by descent $x\in\Gamma^s(R)$.
\end{proof}

Let us point out that the ring $K[\alpha,\beta]\,/(\alpha^{p^n},\beta^{p^{n-r+1}})$ is not free as a module over $K[\lambda]$, which one sees by analyzing the size of the Jordan blocks for multiplication by $\lambda=\alpha\beta^p-\beta\alpha^{p^r}$.  Thus it is not always possible to factor a given element of $\Gamma^{r+1}(R)$ into commutators, even over flat extensions $R\subset R'$.

In the preceding proof, we have used the following   fact. 

\begin{lemma}
\label{higher commutators unitriangular}
The unitriangular matrix group $\operatorname{UT}_{n+1}(R)$, over any ring $R$, has upper central series given by
\begin{equation}
\label{description higher commutator}
Z_s=\left\{ E+ \left(\zeta_{ij}\right) \mid \text{$\zeta_{ij}=0$ whenever $j-i\leq n-s$} \right\}.
\end{equation}
\end{lemma}

\begin{proof}
Over fields, this appears in \cite[Example~16.1.2]{Kargapolov; Merzljakov 1979}. The general case is formulated in \cite{Robinson 1993} as Exercise 5.1.13.  For the sake of completeness, we sketch an argument, by induction on $s\geq 0$.  The case $s=0$ is trivial. Now suppose $s\geq 1$ and that the assertion is true for $s-1$. By definition, a unitriangular $E+(\alpha_{ij})$ belongs to $Z_s$ if and only if
\begin{equation}
\label{higher commutator condition}
\left(E+\left(\alpha_{ij}\right)\right) \cdot \left(E+\left(\beta_{ij}\right)\right)  \equiv \left(E+\left(\beta_{ij}\right)\right) \cdot \left(E+\left(\alpha_{ij}\right)\right) \quad \text{modulo $Z_{s-1}$}
\end{equation}
for every unitriangular matrix $E+(\beta_{ij})$. By the induction hypothesis, each $E+(\zeta_{ij})\in Z_{s-1}$ has $\zeta_{ij}=0$ for $j-i\leq n+1-s$, and one easily checks that right multiplication with elements of $Z_{s-1}$ to a unitriangular matrix leaves the $(i,k)$-entries unchanged for $k-i\leq n+1-s$.  From this the inclusion $\supset$ of \eqref{description higher commutator} easily follows.  Conversely, suppose that we have some $E+(\alpha_{ij})\in Z_s$, so \eqref{higher commutator condition} holds.  This means
$$
\sum_j\alpha_{ij}\beta_{jk} = \sum_j \beta_{ij}\alpha_{jk}\quad\text{whenever $k-i\leq n+1-s$,}
$$
where $E+(\beta_{ij})$ is any unitriangular matrix and the sums actually run over $i<j<k$.  From this one easily infers by induction on $r=j-i$ that $\alpha_{ij}$ vanishes for $1\leq r\leq n-s$.
\end{proof}

\section{Compactifications and numerical semigroups}
\label{Compactifications}

We keep the setting of the previous section but now work with a new indeterminate $T=x^{-1}$.  The iterated semidirect product $\GG_a\rtimes U_n\rtimes \GG_m$ has as coordinate ring
$$
\Gamma\left(\O_{\GG_a\rtimes U_n\rtimes \GG_m}\right)=
K\left[\alpha,\lambda_1,\ldots,\lambda_n,\lambda_0^{\pm}\right]/\left(\lambda_1^{p^n},\lambda_2^{p^{n-1}},\ldots,\lambda_{n-1}^{p^2},\lambda_n^p\right),
$$
endowed with a Hopf algebra structure, and acts on the affine line $\AA^1=\Spec K[T^{-1}]$. We now seek to extend this action to certain compactifications, all of which are denormalizations of the projective line $\PP^1=\Spec K[T]\cup\Spec K[T^{-1}]$. For this, we have to make extensive computations in the first chart, which are much easier to carry out with $T$ rather than $x^{-1}$.  Note that by Proposition~\ref{extension to fractions} we have an induced action on the spectrum of the function field $K(T)=K(x)$, and this action takes the form
\begin{equation}
\label{action via hopf}
K(T)\lra \Gamma\left(\O_{\GG_a\rtimes U_n\rtimes \GG_m}\right)\otimes K(T),\quad T\longmapsto \left(\alpha+\sum_{i=0}^n\lambda_iT^{-p^i}\right)^{-1}.
\end{equation}
 
Recall that an additive submonoid $\Gamma\subset\NN$ whose complement is finite is called a \emph{numerical semigroup}.  Equivalently, the induced inclusions of groups $\Gamma^{\grp}\subset\NN^{\grp}=\ZZ$ is an equality, or $\gcd(a_1,\ldots,a_r)=1$ for some members $a_1,\ldots,a_r\in \Gamma$.  Each numerical semigroup comes with the following invariants: The \emph{multiplicity} $e\geq 1$ is the smallest non-zero element in $\Gamma$.  The \emph{conductor} is the smallest integer $c\geq 0$ with $\{c,c+1,\ldots\}\subset\Gamma$.  The \emph{genus} $g\geq 0$ is the cardinality of the complement $\Gamma\smallsetminus\NN$, whose members are called \emph{gaps}.  As monoid, $\Gamma$ is finitely generated, and among all systems of generators, there is a smallest one; its cardinality is called the \emph{embedding dimension} $d\geq 1$.  For general overviews, we refer to the textbooks \cite{Rosales; Garcia-Sanchez 2009} and \cite{Assi; Garcia-Sanchez 2016}.

For each numerical semigroup $\Gamma$, the ring $K[T^\Gamma]=K[T^a\mid a\in \Gamma]$ defines a \emph{compactification}
$$
X=\Spec K\left[T^\Gamma\right] \cup \Spec K\left[T^{-1}\right]
$$
of the affine line $\AA^1=\Spec K[T^{-1}]$, obtained by adding a single rational point $x_0\in X$. The gluing of the two affine open sets is given by the common localization $K[T^{\pm 1}]$ of the coordinate rings.  The normalization is $\PP^1=\Spec K[T]\cup\Spec K[T^{-1}]$, and the ensuing map $f\colon \PP^1\ra X$ is described by the \emph{conductor square}
\begin{equation}
\label{conductor square}
\begin{CD}
A	@>>>	\PP^1\\
@VVV		@VVfV\\
B	@>>>	X,
\end{CD}
\end{equation}
which is both cartesian and cocartesian (for details see \cite[Appendix~A]{Fanelli; Schroeer 2020}).  The conductor loci $A\subset\PP^1$ and $B\subset X$ are the closed subschemes whose respective coordinate rings are $ K[T]\,/(T^c)$ and $ K[T^\Gamma]\,/(T^c,T^{c+1},\ldots)$.  Consider the short exact sequence $0\ra \O_X\ra f_*(\O_{\PP^1})\times\O_B\ra f_*(\O_A)\ra 0$ of sheaves on $X$, where the inclusion is the diagonal map and the surjection is the difference map. It yields
\begin{equation}
\label{invariants curve}
h^0\left(\O_X\right)=1,\quad  h^1\left(\O_X\right)=g\quadand e\left(\O_{X,x_0}\right)=e,\quad \edim\left(\O_{X,x_0}\right)=d,
\end{equation}
with the invariants $c,g,e,d$ of the numerical semigroup discussed above. Here $e(\O_{X,x_0})$ and $\edim(\O_{X,x_0})$ denote the \emph{multiplicity} and the \emph{embedding dimension} of the local ring, respectively.

Given a subgroup scheme $G\subset\GG_a\rtimes U_n\rtimes \GG_m$, it is natural to ask whether the resulting $G$-action on the affine line $\AA^1$ extends to the compactification $X$.  If it exists, such an extension is unique because the open set $\AA^1\otimes R$ is schematically dense in $X\otimes R$ for any ring $R$.

In the following assertion on the constituents of the iterated semidirect product, we regard the expression $P=(1+\sum_{i=1}^n \lambda_iT^{1-p^i})^{-d}$ as a Laurent polynomial in the indeterminate $T$ with coefficients from $\FF_p[\lambda_1,\ldots,\lambda_n]\,/(\lambda_1^{p^n},\lambda_2^{p^{n-1}},\ldots,\lambda_n^p)$, and $Q=(1+\alpha T)^{-d}$ as a formal power series in $T$ with coefficients from $\FF_p[\alpha]$.  In both cases we use the ensuing notion of supports $\Supp(P)$ and $\Supp(Q)$ inside the group of exponents $\ZZ$.

 
\begin{proposition}
\label{criteria for extensions}
We keep the notation as above. Then the following hold:
\begin{enumerate} 
\item\label{ce-1} 
The   multiplicative group $G=\GG_m$ always admits an extension.
\item\label{ce-2} For the infinitesimal group scheme $G=U_n$, the extension exists if and only if for each $d\in \Gamma$ and $s\in \Supp(P)$ for the Laurent polynomial $P=(1+\sum_{i=1}^n \lambda_iT^{1-p^i})^{-d}$, we also have $d+s\in \Gamma$.\item\label{ce-3} For the additive group $G=\GG_a$, the extension exists if and only if for each $d\in \Gamma$ and $s\in\Supp(Q)$ for the formal power series $Q=(1+\alpha T)^{-d}$, we have $d+s\in\Gamma$.
\end{enumerate} 
Moreover, it suffices to verify these conditions for a set of generators $d\in\Gamma$.
\end{proposition}

\begin{proof}
\eqref{ce-1}~  Recall that  $\GG_m$-actions on affine schemes correspond to $\ZZ$-gradings, according to \cite[Expos\'e I, Corollary 4.7.3.1]{SGA 3a}.
The action on $\Spec K[T^{-1}]$ is given by    $\deg(T^{-i})=-i$. 
This also defines compatible gradings on $K[T^\Gamma]$, which yields the desired extension of the action of $G=\GG_m$.

\eqref{ce-2}~ The group scheme $G=U_n$ is infinitesimal; hence every open set on a $G$-scheme is $G$-stable.  It follows that the $G$-action extends if and only if the map $K[T^\Gamma]\ra \Gamma(\O_G)\otimes K(T)$ induced from \eqref{action via hopf} factors over the subring $\Gamma(\O_G)\otimes K[T^\Gamma]$.  This map sends $T^{-1}$ to $T^{-1} + \sum \lambda_iT^{-p^i}=T^{-1}(1+\sum \lambda_iT^{1-p^i})$.  Note that the second factor is invertible because its second summand is nilpotent.  The monomial $T^d$ with $d\in \Gamma$ is mapped to $T^dP(T)$. This belongs to the subring $R[T^\Gamma]$ if and only if for each $s\in\Supp(P)$, the resulting integer $d+s$ belongs to the numerical semigroup $\Gamma$.

\eqref{ce-3}~ The action of $G=\GG_a$ on $\AA^1=\Spec K[T^{-1}]$ extends to the projective line $\PP^1=\Proj K[U_0,U_1]$ via the assignments $U_1\mapsto U_1$ and $U_0\mapsto U_0+\alpha U_1$, with $T=U_1/U_0$.  Note that the origin $0\in\PP^1$ is fixed but does not admit a stable affine open neighborhood. However, the infinitesimal neighborhoods and in particular the conductor locus $A\subset\PP^1$ are stable.

Since $G$ is smooth, any $G$-action on $\AA^1$ uniquely extends to $\PP^1$, according to \cite[Theorem~2]{Brion 2022a}. By \cite[Lemma~3.5]{Laurent 2019} the $G$-action on the projective line descends to an action on $X$ if and only if the action on the conductor locus $A$ descends to an action on $B$. The latter simply means that the map
\begin{equation}
\label{map to factor}
\Gamma(B,\O_B)\lra \Gamma(A,\O_A)\lra K[\alpha]\otimes\Gamma(A,\O_A)
\end{equation}
factors over $K[\alpha]\otimes\Gamma(B,\O_B)$. Here the map on the right describes the $G$-action on $A$, and the coordinate ring on the left is $\Gamma(B,\O_B)=K[T^\Gamma]\,/(T^c,T^{c+1},\ldots)$, where $c\geq 0$ is the conductor of the numerical semigroup.  As $K$-vector space, this is generated by the residue classes of $T^d$, $d\in \Gamma$.  The map \eqref{action via hopf} sends $T^{-1}$ to $ T^{-1}+\alpha=T^{-1}(1+\alpha T)$, so the monomial $T^d$ is mapped to $T^dQ(T)$.  The class of the latter belongs to $K[T^\Gamma]\,/(T^c,T^{c+1},\ldots)$ if and only if for all $s\in\Supp(Q)$, we have $d+s\in \Gamma$.
\end{proof}
 
Note that in the expansions of $P(T)$ and $Q(T)$, some multinomial coefficients appear, and the above conditions involve their congruence properties modulo the prime number $p$.  Also note that one may view $X$ as a \emph{non-normal torus embedding} with respect to the one-dimensional torus $\GG_m=\Spec K[T^{\pm 1}]$.

The passage from the constituents to the semidirect product is immediate, thanks to the following observation.

\begin{lemma}
\label{semidirect extensions}
Suppose that for each constituent of the iterated semidirect product $G=\GG_a\rtimes U_n\rtimes\GG_m$, the action on $\AA^1$ extends to $X$. Then the  whole $G$-action extends to $X$.
\end{lemma}

\begin{proof}
This is a general fact: all relations between the $R$-valued points of the constituents stemming from the semidirect product structures hold on $\AA^1$ and thus also on $X$ because the former is schematically dense in the latter.
\end{proof}

We now introduce  a particular $\Gamma$ that is generated by $n+1$ numbers. 

\begin{definition}
\label{our semigroup}
We write $\Gamma_{p,n}\subset\NN$ for the numerical semigroup generated by
\begin{equation}
\label{generators semigroup}
p^n\quadand p^n-p^j\quad (0\leq j\leq n-1).
\end{equation}
\end{definition}

This is indeed a numerical semigroup because $\gcd(p^n, p^n-p^0)=1$.  Its multiplicity is given by
$$
e_{p,n}=\begin{cases}
p^{n-1}(p-1)	& \text{if $p^n\geq 3$},\\
1		& \text{else}\\
\end{cases}
$$
because in the first case, the number $p^{n-1}(p-1)$ is smallest among the generators.  Note that $e_{p,n}=1$ is equivalent to $ p^n\leq 2$, whereas $e_{p,n}=2$ means $3\leq p^n\leq 4$.

We came up with the above generators by determining for a handful of special cases the largest numerical semigroup for which the group scheme action extends and then guessing the general pattern. The computations were made with the computer algebra systems Magma \cite{Magma} and GAP \cite{GAP}.  One of the main insights of this paper is that the resulting compactifications
$$
X_{p,n}=\Spec K\left[T^{\Gamma_{p,n}}\right]\cup\Spec K\left[T^{-1}\right]
$$
lead to the desired generalizations of the quasielliptic curves.  Indeed, in the special cases $3\leq p^n\leq 4$, we get $\Gamma_{p,n}=\langle 2,3\rangle$, and the ensuing coordinate rings become $K[T^2,T^3]$. We now verify that the action of the iterated semidirect product extends to this compactification.

\begin{theorem}
\label{action extends}
The action of the group scheme $\GG_a\rtimes U_{p,n}\rtimes\GG_m$ on the affine line $\AA^1=\Spec K[T^{-1}]$ extends to the compactification $X=X_{p,n}$.
\end{theorem}

\begin{proof}
It suffices to extend the action for the three constituents of the iterated semidirect product, by Lemma~\ref{semidirect extensions}, and for this we use Proposition~\ref{criteria for extensions}: The case $G=\GG_m$ is immediate.  Now suppose $G=U_n$, and fix one of the generators $d\in\Gamma_{p,n}$ listed in \eqref{generators semigroup}. We have to understand the expression
$$
P=\left(1+\sum_{i=1}^n \lambda_iT^{1-p^i}\right)^{-d}.
$$
In the case $d=p^n$, the above simplifies to $P=1^{-1}=1$, by the multinomial theorem and $\lambda_i^{p^n}=0$. Thus $\Supp(P)=\{0\}$, and obviously $d+0\in\Gamma_{p,n}$.  In the case $d=p^n-p^j$ with $0\leq j\leq n-1$, we get
$$
P=\left(1+\sum_{i=1}^n \lambda_iT^{1-p^i}\right)^{-p^n}\left(1+\sum_{i=1}^n \lambda_iT^{1-p^i}\right)^{p^j}=1+\sum_{i=1}^n\lambda_i^{p^j} T^{p^j-p^{i+j}}.
$$
Its support equals the set $\{0\}\cup\{p^j-p^{i+j}\mid 1\leq i\leq n-j\}$, in light of the defining relations $\lambda_i^{p^{n-i+1}}=0$.  Obviously, $d+0=p^n-p^j$ and $d+(p^j-p^{i+j})=p^n-p^{i+j}$ belong to $\Gamma_{p,n}$.  Thus the action of $G=U_n$ extends.

It remains to treat the case $G=\GG_a$. Again we fix one of the generators $d\in\Gamma_{p,n}$ and now have to examine the formal power series $Q=(1+\alpha T)^{-d} $ with coefficients from the polynomial ring $\FF_p[\alpha]$. For $d=p^n-p^j$, this becomes
$$
Q=(1+\alpha T)^{p^j}/(1+\alpha T)^{p^n} = \left(1+\alpha^{p^j}T^{p^j}\right)\sum_{i=0}^\infty(-\alpha T)^{ip^n}.
$$
The support is contained in $\{ip^n\mid i\geq 0\}\cup\{p^j+ip^n\mid i\geq 0\}$.  Clearly, $d+ip^n=(p^n-p^j) + ip^n$ and $d+(p^j+ip^n)=(i+1)p^n$ belong to $\Gamma_{p,n}$.  The argument for $d=p^n$ is likewise, and even simpler. Thus the action of $G=\GG_a$ extends.
\end{proof}

Set $\Gamma=\Gamma_{p,n}$ and $X=X_{p,n}$.  With respect to the infinitesimal group scheme $U_n$, all open sets in $X$ are stable, and the action on the affine open set $ \Spec K[T^\Gamma]$ is given by the ring homomorphism
$$
K[T^\Gamma]\lra \Gamma\left(\O_{U_n}\right)\otimes K\left[T^\Gamma\right],\quad T^d\longmapsto T^d\left(1+\sum_{i=1}^n\lambda_iT^{-p^i}\right)^{-d}
$$ 
with exponents $d\in\Gamma$.  The \emph{orbit map} $x_0\colon U_n\ra X$ corresponding to the rational point $x_0\in X$ is given by the homomorphism $\varphi\colon K[T^\Gamma]\ra\Gamma(\O_{U_n})$ that is implicitly described by
$$
\varphi\left(T^d\right)=  T^d\left(1+\sum_{i=1}^n\lambda_iT^{1-p^i}\right)^{-d} \bigg|_{T=0}.
$$
Note that one has to determine the product before substituting $T=0$ because the second factor usually contains terms of negative degree.  The computation for the generators \eqref{generators semigroup} of our numerical semigroup is immediate: $\varphi(T^{p^n})=0$ and $\varphi(T^{p^n-p^j})=\lambda_{n-j}^{p^j} $ for $0\leq j\leq n-1$. Now recall that the \emph{inertia group scheme} in $U_n$ is defined by the largest quotient of $\Gamma(\O_{U_n})$ in which $\varphi$ becomes the zero map. Setting $i=n-j$, we get the following.

\begin{proposition}
\label{inertia}
Inside $U_n=\Spec K[\lambda_1,\ldots,\lambda_n]\,/(\lambda_1^{p^n},\lambda_2^{p^{n-1}},\ldots,\lambda_n^p)$, the inertia group scheme with respect to the rational point $x_0\in X$ is defined by the equations $\lambda_i^{p^{n-i}}=0$ for $1\leq i\leq n$.
\end{proposition}

This inertia group scheme coincides with the canonical inclusion of $U_{n-1}\subset U_n$, which is also the image of the relative Frobenius map, and we thus obtain a $U_n$-stable closed subscheme $U_n/U_{n-1}\subset X$. \textit{A priori}, this is an effective Weil divisor supported by $x_0$, of degree $[U_n:U_{n-1}] = h^0(\O_{U_n})/h^0(\O_{U_{n-1}})= p^n$.  The following observation will be crucial in what follows.

\begin{proposition}
\label{cartier divisor}
The  Weil divisor    $U_n/U_{n-1}\subset X$ is an effective Cartier divisor.
\end{proposition}

\begin{proof}
The closed subscheme lies in the affine open set $\Spec K[T^\Gamma]$ and corresponds to the ideal $\ideala=\Kernel(\varphi)$. This ideal contains the monomial $T^{p^n}$, and we claim that the inclusion $(T^{p^n})\subset\ideala$ is an equality.  In other words, we have to verify that the resulting map
$$
\varphi\colon K\left[T^\Gamma\right]/\left(T^{p^n}\right)\lra \Gamma(\O_G)= K[\lambda_1,\ldots,\lambda_n]\,/\left(\lambda_1^{p^n},\lambda_2^{p^{n-1}},\ldots,\lambda_n^p\right)
$$
is injective. We computed above that its image is the subring generated by the powers $\lambda_i^{p^{n-i}}$ for $1\leq i\leq n$, which is a $K$-algebra of degree $p^n$.  So it suffices to verify that the $K$-algebra $K[T^\Gamma]\,/(T^{p^n})$ has degree at most~$p^n$.  This algebra is generated by the classes $x_j$ of $T^{p^n-p^j}$ with $0\leq j\leq n-1$.  From the relation
$$
p(p^n-p^j)= (p-1)p^n + \left(p^n-p^{j+1}\right)  
$$
in the numerical semigroup $\Gamma$, we infer a factorization $(T^{p^n-p^j})^p= (T^{p^n})^{p-1}\cdot T^{p^n-p^{j+1}}$ in the ring $K[T^\Gamma]$, and hence $x_j^p=0$. Thus $K[T^\Gamma]\,/(T^{p^n})$ has degree at most $p^n$.
\end{proof}

\section{The complete intersection property}
\label{Complete intersection}

We keep the notation as in the preceding section and continue to study the algebra of the numerical semigroup $\Gamma=\Gamma_{p,n}$ and also the geometry of the compactification $X=X_{p,n}$ of the affine line $\AA^1=\Spec K[T^{-1}]$ defined by the coordinate ring $K[T^\Gamma]$.

Recall that any numerical semigroup $\Gamma$ given by a set of $d\geq 1$ generators $a_1,\ldots,a_d$ and ensuing surjection $\NN^d\ra \Gamma$ is called a \emph{complete intersection} if the congruence $R=\NN^d\times_\Gamma\NN^d$ is generated by $d-1$ elements.  According to \cite[Corollary 1.13]{Herzog 1970}, this is equivalent to the condition that the complete local ring $A=K[[T^\Gamma]]$ is a complete intersection in the sense of commutative algebra; in other words, $A\simeq K[[u_1,\ldots,u_r]]\,/(f_1,\ldots,f_s)$ for some $r\geq 0$ and some regular sequence $f_1,\dots,f_s$, here necessarily with $s=r-1$.

\begin{proposition}
\label{conductor}
Our numerical semigroup $\Gamma_{p,n}$ is a complete intersection, and its conductor $c_{p,n}$ and genus $g_{p,n}$ are given by the formulas
$$
c_{p,n}=np^{n+1}-(n+2)p^n+2 \quadand g_{p,n}=\tfrac{1}{2}c_{p,n}.
$$
Moreover, $G=\{p^n-p^{n-1}, p^n-p^{n-2}, \ldots, p^n-1, p^n\}$ is the smallest generating set provided $p\geq 3$; for the prime $p=2$ and $n\geq 1$, one has to omit $p^n$.
\end{proposition}

\begin{proof}
First note that for $p=2$ and $n\geq 1$, the relation $p^n=2(p^n-p^{n-1}) $ shows that the generator $p^n$ does not belong to the smallest generating set.

We now proceed, for general $p>0$, by induction on $n\geq 0$.  For $n=0$, we have $\Gamma=\NN$, and all assertions are obvious.  Now suppose $n\geq 1$ and that the assertion holds for $n-1$.  Consider the sets of numbers
$$
G_1=\{p^{n-1}-p^{n-2},  p^{n-1}-p^{n-3},\ldots,  p^{n-1}-1,   p^{n-1}\}\quadand G_2=\{1\}.
$$ 
Both generate respective numerical semigroups $\Gamma_1$ and $\Gamma_2$, and the induction hypothesis applies to the former.  The numbers $a_1=p$ and $a_2=p^n-1$ are relatively prime, with $a_1\in \Gamma_2$ and $a_2=p(p^{n-1}-p^{n-2})+(p^{n-1}-1)\in\Gamma_1$.  Furthermore, $\Gamma=a_1\Gamma_1+a_2\Gamma_2$.  According to \cite[Proposition~10]{Delorme 1976}, the monoid $\Gamma$ is a complete intersection, and the conductor is given by the formula
\begin{equation}
\label{c formula}
c=a_1c_1 + a_2c_2 + (a_1-1)(a_2-1) = pc_1  + (p-1)(p^n-2).
\end{equation}
Here $c_2=0$ is the conductor of $\Gamma_2$, and $c_1=(n-1)p^{n}-(n+1)p^{n-1}+2$ is the conductor of $\Gamma_1$, which we know by our induction hypothesis.  Inserting the latter into \eqref{c formula}, we get the desired formula for $c_{p,n}$.  Every complete intersection semigroup is symmetric (\textit{cf.} \cite[Corollary~9.12]{Rosales; Garcia-Sanchez 2009}), which simply means that the conductor is twice the genus, and the formula for $g_{p,n}$ follows.

Now suppose $p\geq 3$. By induction, the $G_i\subset\Gamma_i$ are the smallest generating sets.  The number $a_1a_2=p^{n+1}-p$ does not belong to $a_1G_1\cup a_2G_2=G$. As explained in \cite[proof of Proposition~10(ii)]{Delorme 1976}, the subset $G\subset\Gamma$ is the smallest generating set.  For $p=2$, one argues likewise, with $p^n$ omitted.
\end{proof}

We see that the embedding dimension for the numerical semigroup $\Gamma_{p,n}$ and the local ring $\O_{X,x_0}$ is given by the formula
$$
d_{p,n}=\begin{cases}
n+1	& \text{if $p\geq 3$ or $n=0$},\\
n	& \text{if $p=2$ and $n\geq 1$}.
\end{cases}
$$
Let us also record the following geometric consequences. 

\begin{corollary}
\label{genus curve}
The curve $X=X_{n,p}$ has invariants 
$$
h^0(\O_X)=1\quadand h^1(\O_X)=\tfrac{1}{2}(np^{n+1}-(n+2)p^n+2).
$$
Moreover, the dualizing sheaf $\omega_X$ is invertible, of degree $p^n(np-n-2)$.
\end{corollary}

\begin{proof}
The values for $h^i(\O_X)$ follow with \eqref{invariants curve} from Proposition~\ref{conductor}.  Being locally of complete intersection, $X$ must be Gorenstein, and the dualizing sheaf is invertible.  Serre duality gives $\deg(\omega_X)=-2\chi(\O_X)=p^n(np-n-2)$.
\end{proof}

We actually can derive an explicit description for the ring $K[T^\Gamma]$ in terms of generators and relations.  Write $a_j=p^n-p^j$ and $b=p^n$ for the generators of $\Gamma=\Gamma_{p,n}$. They give rise to a surjection $\NN^{n+1}\ra \Gamma$ of monoids and an ensuing congruence $R=\NN^{n+1}\times_\Gamma\NN^{n+1}$.  The $n+1$ generators satisfy the $n$ obvious relations
\begin{equation}
\label{obvious relations}
p\cdot a_{n-1} = (p-1)\cdot b\quadand p\cdot a_j = p\cdot a_{n-1} + a_{j+1}\quad (0\leq j\leq n-2),
\end{equation}
which may be interpreted as members of the congruence $R$. To translate this into commutative algebra, let $x_j,y$ be indeterminates corresponding to the generators $a_j,b\in\Gamma$, and consider the surjection
$$
\varphi\colon K[x_1,\ldots,x_{n-1},y]\lra K[T^\Gamma]
$$ 
given by $\varphi(x_j)=T^{a_j}$ and $\varphi(y)=T^b$.  The map respects the gradings specified by $\deg(x_j)=a_j$, $\deg(y)=b$, and $\deg(T)=1$.

\begin{proposition}
\label{homogeneous ideal}
The ideal $\ideala=\Kernel(\varphi)$ is generated by the polynomials $x_{n-1}^p-y^{p-1}$ and $x_j^p-x_{n-1}^px_{j+1}$ for $ 0\leq j\leq n-2$, corresponding to the obvious relations \eqref{obvious relations}.
\end{proposition}

\begin{proof}
This is an application of an observation of Delorme \cite[Lemma~8]{Delorme 1976}. Recall that our numerical semigroup is generated by the $n+1$ elements $a_0,\ldots,a_{n-1},b\in\Gamma$.  Delorme's observation hinges on two descending sequences
$$
P_{n+1},P_n,\ldots, P_1\quadand Z_{n+1},Z_n,\ldots,Z_2.
$$
The first sequence comprises \emph{partitions} $P_i$ of the generating set $G=\{a_0,\ldots,a_{n-1},b\}$, subject to the following condition: $P_{n+1}$ is the partition into singletons, and each $P_{i-1}$ is obtained from its precursor $P_i$ by replacing certain members $L_i,L'_i\in P_i$ by their union.  The second sequence consists of \emph{homogeneous polynomials} $Z_i$ in the indeterminates $x_0,\ldots,x_{n-1},y$, taking the form $Z=H_i- H'_i$ for some monic monomials $H_i$ and $H'_i$, each involving only indeterminates indexed by $L_i$ and $L'_i$, respectively.  In \textit{loc.\ cit.}~the sequences are denoted by $\shP$ and $\shZ$, and the pair $(\shP,\shZ)$ is called a \emph{suite distingu\'ee}.

Note that the partitions $P_i$ are fully determined by the sets $L_i,L'_i\subset G$ with $2\leq i\leq n+1$.  We now define such a partition sequence by setting
$$
L_i=\{a_{i-1},\ldots,a_{n-1},b\}\quadand L'_i=\{a_{i-2}\}.
$$
Note that this starts with the singletons $L_{n+1}=\{b\}$ and $L'_{n+1}=\{a_{n-1}\}$. The homogeneous polynomials are declared as
$$
Z_{n+1}=y^{p-1}-x_{n-1}^p\quadand Z_i=x_{i-2}^p-x_{n-1}^px_{i-1}\quad (2\leq i\leq n).
$$
These have $\deg(Z_i)= p^{n+1}-p^{i-1}$ for all $2\leq i\leq n+1$. One sees
$$
\gcd(L_i) =\gcd\left(p^{i-1}, \ldots, p^{n-1}, p^n\right)= p^{i-1} \quadand \gcd\left(L'_i\right)= p^n-p^{i-2}.
$$
The least common multiple of the above two gcds is given by $p(p^n-p^{i-2})$, which coincides with $\deg(Z_i)$.  Our assertion now follows from \cite[Lemma~8]{Delorme 1976}.
\end{proof}

This has important consequences for K\"ahler differentials.

\begin{corollary}
\label{tangent sheaf}
The sheaf $\Omega^1_{X/K}/\Torsion$ is invertible of degree $-p^n$, and the tangent sheaf $\Theta_{X/K}=\uHom(\Omega^1_{X/K},\O_X)$ is invertible of degree $p^n$.
\end{corollary}

\begin{proof}
The main task is to compute the module of K\"ahler differentials for the integral domain $K[T^\Gamma]$.  In light of Proposition~\ref{homogeneous ideal}, $\Omega^1_{K[T^\Gamma]/K}$ is generated by the $n+1$ differentials $dx_j$ and $dy$, modulo the $n$ relations
\begin{equation}
\label{relation differentials}
y^{p-2}dy\quadand x_{n-1}^pdx_{j+1} \quad (0\leq j\leq n-2).
\end{equation}
The ring elements $y$ and $x_{n-1}$ are non-zero because they correspond to monomials in $K[T^\Gamma]$, so $dy$ and $dx_{j+1}$ for $0\leq j\leq n-2$ are torsion.  We infer that the map $K[T^\Gamma]\ra \Omega^1_{K[T^\Gamma]/K}$ given by the remaining differential $dx_0$ is bijective modulo torsion.  The latter differential is given by $dT^{p^n-1}$.

Let $\shN$ be the quotient of $\Omega^1_{X/K}$ by its torsion subsheaf, and consider the affine open covering $X=U_0\cup U_1$ with $U_0=\Spec K[T^\Gamma]$ and $U_1=\Spec K[T^{-1}]$. We have trivializations $\shN|U_0$ and $\shN|U_1$, given by $dT^{p^n-1}$ and $dT^{-1}$. On the overlap these become $-T^{p^n-2}dT$ and $-T^{-2}dT$, which are related by the cocycle $T^{p^n}\in\Gamma(U_0\cap U_1,\O_X^\times)$.  This gives $\deg(\shN)=-p^n$. The assertion for the dual sheaf $\Theta_{X/K}=\shN^\vee$ is immediate.
\end{proof}

\section{The projective model}
\label{Projective model}

We keep the set-up of the previous section and now describe a projective model for our curve $X=X_{p,n}$.  First note that the $n$ obvious relations \eqref{obvious relations} for our monoid $\Gamma=\Gamma_{p,n}$ can be replaced by
\begin{equation}
\label{more obvious relations}
p\cdot a_{n-1} = (p-1)\cdot b\quadand p\cdot a_j = (p-1)\cdot b + a_{j+1}\quad (0\leq j\leq n-2)
\end{equation}
by using the first of these relations. Now write $\PP^{n+1}=\Proj K[U_0,\ldots,U_{n-1}, V,Z]$, and consider the closed subscheme $C=C_{p,n}$ defined by the $n$ homogeneous equations
\begin{equation}
\label{homogeneous equations}
U_{n-1}^p-V^{p-1}Z=0\quadand U_j^p-V^{p-1}U_{j+1}=0\quad (0\leq j\leq n-2).
\end{equation}
First observe that $C$ is covered by $D_+(Z)\cup D_+(V)$ because it contains only the point $(0:\ldots:0:1:0) $ on the hyperplane given by $Z=0$.  On these two charts, we see that
$$
T^{p^n-p^j}\longmapsto U_j/Z,\quad T^{p^n}\longmapsto V/Z,\quadand T^{-1}\longmapsto U_0/V
$$
constitute  an isomorphism   $C\ra X$, which we regard as an identification.  

\begin{proposition}
\label{tangent sheaf very ample}
The homogeneous polynomials \eqref{homogeneous equations} form a regular sequence in the polynomial ring, the curve $X\subset\PP^{n+1}$ has degree $p^n$, and
$$
\omega_X=\O_X(np-n-2)\quadand \Theta_{X/k}=\O_X(1).
$$
In particular, $\Theta_{X/k}$ is very ample, and $\omega_X=\Theta_{X/k}^{\otimes r}$ with exponent $r=np-n-2$.
\end{proposition}
 
\begin{proof}
Let $\ideala$ be the ideal generated by the $n$ homogeneous polynomials \eqref{homogeneous equations} inside the $(n+2)$-dimensional Cohen--Macaulay ring $A=K[U_0,\ldots,U_{n-1}, V,Z]$.  Since the scheme $C$ is one-dimensional, we must have $\dim(A/\ideala)=2$.  It follows from \cite[Tag 02JN]{SP} that the polynomials in question form a regular sequence.  The assertion on the dualizing sheaf immediately follows from $\omega_{\PP^{n+1}}=\O_{\PP^{n+1}}(-n-2)$ and the adjunction formula.

The intersection of $C\subset\PP^{n+1}$ with the hyperplane given by $V=0$ is a singleton, with generators $U_j/Z$ and relations $(U_j/Z)^p=0$ for $0\leq j\leq {n-1}$ in the homogeneous coordinate ring.  Thus $\deg(C)=p^n$.

It remains to verify the statement on the tangent sheaf.  As described in the last paragraph of the proof for Corollary~\ref{tangent sheaf}, the invertible sheaf $\Theta_{X/K}$ is given by the cocycle $T^{-p^n}$ with respect to the open covering $W_0= \Spec K[T^\Gamma]$ and $W_1=\Spec K[T^{-1}]$.  The latter correspond to the open sets $D_+(Z)$ and $D_+(V)$. On the union of these open sets, the invertible sheaf $\O_{\PP^n}(1)$ is defined by the cocycle $Z/V$. This becomes $T^{-p^n}$ after restricting to $C$, and thus $\Theta_{X/k}=\O_X(1)$.
\end{proof}

Recall that a square root for the dualizing sheaf is called a \emph{theta characteristic} or \emph{spin structure} (\textit{cf.} \cite{Atiyah 1971} and \cite{Mumford 1971}).  In our situation, the curve $X$ comes with what one might call an \emph{$r$-fold} theta characteristic or spin structure.

Another highly relevant consequence: the very ample sheaf $\O_X(1)=\Theta_{X/K}$ has an intrinsic meaning, and $\lieg=H^0(X,\O_X(1))$ becomes the Lie algebra for the automorphism group scheme $G=\Aut_{X/K}$. To exploit this, we check that the closed embedding $X\subset \PP^{n+1}$ is defined by the complete linear system.

\begin{proposition}
\label{linear system}
The restriction map $H^0(\PP^{n+1},\O_{\PP^{n+1}}(1))\ra H^0(X,\O_X(1))$ is bijective.  In particular, $h^0(\Theta_{X/K})=n+2$.
\end{proposition}

\begin{proof}
Since the defining polynomials \eqref{homogeneous equations} have degree $p\geq 2$, the homogeneous ideal for $X\subset\PP^{n+1}$ contains no linear terms. It follows that the map in question is injective.  It remains to compute $h^0(\shL)$ for $\shL=\Theta_{X/K}$.

Let us proceed with some general considerations on invertible sheaves $\shL$ on $X$ of arbitrary degree $m\geq 0$.  Recall that the conductor loci for the normalization map $f\colon \PP^1\ra X$ are given by
$$
\Gamma(\O_A)= K[T]\,/(T^c)\quadand \Gamma(\O_B)=K[T^\Gamma]\,/(T^c,T^{c+1},\ldots),
$$
where $c\geq 0$ is the conductor for the numerical semigroup $\Gamma$. From the cocartesian diagram \eqref{conductor square}, we now obtain an exact sequence
$$
0\lra H^0(\shL)\lra \Gamma(\shL_{\PP^1})\oplus \Gamma(\shL_B)\lra \Gamma(\shL_A)\lra H^1(\shL)\lra 0.
$$
It is not difficult to determine the map in the middle: Making the identification $\shL_{\PP^1} =\O_{\PP^1}(m)$ and $\shL_B=\O_B$ and $\shL_A=\O_A$, we get
$$
\Gamma(\shL_{\PP^1})=\bigoplus_{a=0,\ldots,m}KT^a,\quad \Gamma(\shL_B)=\bigoplus_{a\in\Gamma, a\leq c-1}KT^a,\quadand 
\Gamma(\shL_A)= \bigoplus_{a=0,\ldots,c-1}KT^a.
$$
If $m\leq c-1$, the former groups are contained in the latter, and $H^0(\shL)$ becomes their intersection, and hence $h^0(\shL)=\Card(S)$ for the set
$$
S=\{a\in\Gamma\mid \text{$a\leq m$ and $a\leq c-1$}\} = \{a\in\Gamma\mid \text{$a\leq m$}\}.
$$
We have to determine this set for $m=p^n$, under the assumption $n(p-1)\geq 3$.  According to Proposition~\ref{conductor}, the conductor is $c=np^{n+1}-(n+2)p^n+2$, and thus $c/p^n = np-(n+2) +2/p^n > n(p-1) -2\geq 1$.  So our set $S$ comprises all $a\in \Gamma$ with $a\leq p^n$.  It clearly contains the generators $p^n,p^n-1,\ldots,p^n-p^{n-1}$ and also the zero element $a=0$.  It remains to check that for each pair of generators $a\leq b$, we have $a+b\geq p^n$. This is obvious for $b=p^n$, so assume $a=p^n-p^i$ and $b=p^n-p^j$ with $0\leq i\leq j<n$. Then $a+b-p^n=p^n-p^i-p^j \geq p^n-2p^j\geq p^n-p^{j+1}\geq 0$.
\end{proof}

This leads to a matrix interpretation of the full automorphism group scheme $G=\Aut_{X/K}$: First note that the diagonal action of $G$ on $X\times X$, and its effects on graphs, induces the conjugacy action of $G$ on itself. Its restriction to the first infinitesimal neighborhood of the diagonal $\Delta_X$ yields the $G$-linearization of the tangent sheaf $\Theta_{X/k}$, and we infer that the resulting representation on the Lie algebra $\lieg=H^0(X,\Theta_{X/K})$ coincides with the adjoint representation $G\ra\GL(\lieg)$.  Its projectivization $G\ra\PGL(\lieg)$ is injective because $\Theta_{X/k}$ is very ample, and it follows that $G\ra\GL(\lieg)$ is injective as well.  We thus have a canonical inclusion $G\subset\GL(\lieg)$ that intersects the center $\GG_m\subset\GL(\lieg)$ trivially.  Write $G\cdot\GG_m=G\times\GG_m$ for the resulting subgroup scheme and $\ideala_p\subset\Sym^p(\lieg)$ for the vector subspace generated by the homogeneous polynomials \eqref{homogeneous equations}.

\begin{proposition}
\label{stabilizer group scheme}
We keep the notation as above. Then $G\cdot\GG_m\subset \GL(\lieg)$ equals the stabilizer group scheme for the vector subspace $\ideala_p\subset \Sym^p(\lieg)$.
\end{proposition}

\begin{proof}
We start with some observations on the homogeneous coordinate rings
$$
\Gamma_\bullet(\O_{\PP^{n+1}})  = \Sym^\bullet(\lieg)\quadand \Gamma_\bullet(\O_X) =\bigoplus_{n\geq 0} \Gamma(X,\O_X(n)).
$$
Both rings are integral, so the kernel $\primid$ of the canonical map $\Gamma_\bullet(\O_{\PP^{n+1}})\ra \Gamma_\bullet(\O_X)$ is a prime ideal, which equals the radical for the ideal $\ideala$ generated by the polynomials \eqref{homogeneous equations}.  The ideal $\ideala$ becomes prime when localized with respect to any homogeneous $f\in\Sym^+(\lieg)$ because $X$ is integral.  Since our generators form a regular sequence, this actually holds everywhere, and thus $\ideala=\primid$.

Write $I\subset \GL(\lieg)$ for the stabilizer group scheme in question. It contains $\GG_m$ because the polynomials are homogeneous, and its action on $\PP^{n+1}$ stabilizes the curve $X$. Modulo $\GG_m$, the induced action on $X$ is faithful, according to Proposition~\ref{linear system}. Thus $I\subset G\cdot \GG_m$.  Conversely, let $f\in G(R)$ be some $R$-valued automorphism of $X$.  It induces an action on the homogeneous coordinate rings $\Gamma_\bullet(\O_{\PP^{n+1}}\otimes R)=\Gamma_\bullet(\O_{\PP^{n+1}})\otimes R$, and likewise for $\Gamma_\bullet(\O_X)$. These actions are compatible; thus $f$ stabilizes $\primid_p\otimes R=\ideala_p\otimes R$.
\end{proof}

For $n=1$, this means that $G\cdot \GG_m\subset\GL_3$ is the stabilizer group scheme for the line generated by $U^p-V^{p-1}Z$ in the $\supth{p}$ symmetric power of $\lieg=KU \oplus KV\oplus KZ$.  In turn, $G\subset\PGL_3$ is the inertia group scheme for the rational point corresponding to this line.

\section{The automorphism group scheme}
\label{Automorphism group}

Recall that our curves $X=X_{p,n}$ come  with an  inclusion
\begin{equation}
\label{inclusion}
\GG_a\rtimes U_n\rtimes \GG_m\subset\Aut_{X/K}
\end{equation}
of  group schemes. The following is one of  the main results of this paper.

\begin{theorem}
\label{automorphism group scheme}
The above inclusion of group schemes is an equality provided $p^n\geq 3$.
\end{theorem}

The cases $p^n\leq 2$ indeed have to be excluded because then $X=\PP^1$. Also note that for $3\leq p^n\leq 4$, our curve $X$ becomes the rational cuspidal curve, and the assertion was established by Bombieri and Mumford \cite[Proposition~6]{Bombieri; Mumford 1976}.  The proof for the above theorem requires some preparation and will be given stepwise. We start with a simple observation.

\begin{lemma}
\label{inclusion bijective}
For $p^n\geq 3$, the ideal for the closed embedding \eqref{inclusion} is nilpotent.
\end{lemma}

\begin{proof}
For this, we may assume that $K$ is algebraically closed.  Seeking a contradiction, we suppose that there is an automorphism $\varphi\colon X\ra X$ that does not yield a rational point in $\GG_a\rtimes\GG_m$. The assumption ensures $\Sing(X)=\{x_0\}$, and $\varphi$ fixes this singular point.  Hence the induced automorphism on the normalization $\PP^1=\Spec K[T]\cup\Spec K[T^{-1}]$ fixes the point defined by $T=0$.  It thus belongs to the inertia group scheme $\GG_a\rtimes\GG_m$ inside $\PGL_2$.  According to \cite[Proposition~2.5.1]{Brion 2017}, the action of any smooth group scheme on $X$ lifts to an action on the normalization.  Thus $\varphi$ belongs to $\GG_a\rtimes\GG_m$ inside $\Aut_{X/k}$, giving a contradiction.
\end{proof}
 
\begin{proof}[Proof of Theorem~\ref{automorphism group scheme} in the special case $n=1$]
In this situation we have $X=\Spec K[T^p,T^{p-1}]\cup\Spec K[T^{-1}]$.  According to Proposition~\ref{linear system}, the tangent sheaf has $h^0(\Theta_{X/k})=3$. One easily computes that the rational vector fields
\begin{equation}
\label{three vector fields}
T^{2-p}\frac{\partial}{\partial T}, \quad  T^2\frac{\partial}{\partial T},	\quadand  T\frac{\partial}{\partial T}
\end{equation}
are everywhere defined, hence form a basis of $\lieg=H^0(X,\Theta_{X/k})$.  Moreover, the first two basis vectors generate a restricted subalgebra $ K^2$, with trivial bracket and $p$-map, and the last basis vector yields a copy of $\gl_1(K)$, giving a semidirect product $K^2\rtimes\gl_1(K)$.  Such algebras play a prominent role in \cite{Schroeer 2007,Kondo; Schroeer 2021,Schroeer; Tziolas 2023}.  The bracket and $p$-map are given by $[(x,\lambda),(x',\lambda')]= \lambda x'-\lambda'x$ and $(x,\lambda)^{[p]} = (\lambda^px, \lambda^{p-1})$; \textit{cf.} \cite[Proposition~1.1]{Kondo; Schroeer 2021}. One sees that $K^2$ is the image of the bracket, thus the derived subalgebra. This also holds for $R$-valued points; hence $K^2$ is a subrepresentation for $G$.

As explained in Section~\ref{Projective model}, our curve $X$ may also be regarded as the curve in $\PP^2=\Proj K[U_0,V,Z]$ defined by the homogeneous polynomial $P=U_0^p-V^{p-1}Z$, with an identification via $T^{p-1}=U_0/Z$, $T^p=V/Z$, and $T^{-1}=U_0/V$. According to Proposition~\ref{linear system}, the monomials $Z,V,U_0$ yield a basis for $H^0(X,\O_X(1))$.  By Proposition~\ref{tangent sheaf very ample}, the invertible sheaves $\O_X(1)$ and $\Theta_X$ are isomorphic.  Computing the order of zeros for the homogeneous polynomials $Z,V,U_0$ and the vector fields \eqref{three vector fields} on $\PP^1$, one sees that each identification $\O_X(1)=\Theta_X$ sends the former basis to the latter basis, at least up to a diagonal base-change matrix.

Combining this with Proposition~\ref{stabilizer group scheme}, we have a functorial interpretation of the $R$-valued points of $H=G\cdot \GG_m$ as the group of matrices
$$
A=\begin{pmatrix}
a&b&c\\
d&e&f\\
0&0&g
\end{pmatrix}\in\GL_3(R)
$$
subject to the sole condition
\begin{equation}
\label{multiplier}
P(aU_0+ dV, bU_0+eV,cU_0+fV+gZ)=\lambda\cdot P(U_0,V,Z),
\end{equation}
with some multipliers $\lambda\in R^\times$. The zero entries in the matrix $A$ stem from the fact that the derived subalgebra $[\lieg,\lieg]\subset\lieg$ is a subrepresentation.

Suppose that $A\in H(R)$ has $b=0$. Comparing coefficients in \eqref{multiplier}, we get $ae^{p-1}=g^p$, $c^p=0$, and $de=f^p$, so the matrix takes the form
$$
A=\begin{pmatrix}
g^p e^{1-p}&0&c\\
f^pe^{-1}&e&f\\
0&0&g
\end{pmatrix} \quad \text{with $c^p=0$.}
$$
The group $H_0\subset H$ of such matrices contains the diagonal copy of $\GG_m$, and $H_0/\GG_m$ becomes our iterated semidirect product $\GG_a\rtimes\alpha_p\rtimes\GG_m$ inside $\Aut_{X/k}\subset\PGL_3$.  Note that the projection $H_0\ra H_0/\GG_m$ admits a splitting, obtained by setting $g=1$.

Seeking a contradiction, we assume that there is some $A\in H(R)$ with $b\neq 0$.  By Lemma~\ref{inclusion bijective}, we must have $b\in \Nil(R)$, and thus $a,e\in R^\times$. Making a flat extension of $R$, we can assume that there is some $f'\in R$ with $f'^p=-d/a$.  Setting $a'=e'=g'=1$ and $b'=c'=0$ and left-multiplying with the resulting matrix $A'\in H(R)$, we may assume that both $b\neq 0$ and $d=0$ hold. On the other hand, comparing coefficients in \eqref{multiplier} immediately yields $b=0$, giving a contradiction.

This establishes $H_0=H$. We already observed that $H/\GG_m=\Aut_{X/K}$ inside $\PGL_3$ and that $H_0/\GG_m$ equals our iterated semidirect product.
\end{proof}

To continue inductively, we seek to relate the curves $X_{p,n}$ with different indices $n$.  First recall a general fact on numerical semigroups $\Gamma=\langle b,a_2,\ldots,a_r\rangle$ with non-zero generators $b<a_2<\ldots<a_r$: the \emph{blowing-up} $\Bl_\maxid(A)$ of the ring $A= K[T^\Gamma]$ with respect to the maximal ideal $\maxid=(T^b,T^{a_2},\ldots,T^{a_r})$ has coordinate ring $A'=K[T^{\Gamma'}]$ for the numerical semigroup $\Gamma'=\langle b,a_2-b,\ldots,a_r-b\rangle$; \textit{cf.} \cite[Proposition I.2.1, and also Equation I.2.4]{Barucci; Dobbs; Fontana 1997}.

For our $\Gamma=\Gamma_{p,n}$ with $n\geq 1$, we have $b=p^n-p^{n-1}$, with remaining generators $p^n$ and $p^n-p^j$ for $0\leq j\leq n-2$.  The resulting differences are $p^n-b = p^{n-1}$ and $(p^n-p^j)-b = p^{n-1}-p^j$.  For $n\geq 2$, we write $b=p\cdot (p^{n-1}-p^{n-2})$ and in any case see $\Gamma'=\Gamma_{p,n-1}$.  This reveals the following.

\begin{lemma}
\label{blowing-up}
For $p^n\geq 3$, we have $X_{p,n-1}=\Bl_Z(X_{p,n})$, where the center $Z$ is the singular point $x_0\in X_{p,n}$ endowed with the reduced scheme structure.
\end{lemma}

Note that for every $m\leq n$, we get an inclusion $\Gamma_{p,n}\subset\Gamma_{p,m}$ of numerical semigroups inside $\NN$.  The resulting inclusions of coordinate rings $K[T^{\Gamma_{p,n}}]\subset K[T^{\Gamma_{p,m}}]$ define canonical morphisms $X_{p,m}\ra X_{p,n}$ of compactifications of the affine line $\AA^1=\Spec K[T^{-1}]$.

\begin{proof}[Proof of Theorem~\ref{automorphism group scheme} in the general case]
We proceed by induction on $n\geq 1$. The case $n=1$ was handled above. Now suppose $n\geq 2$ and that the assertion is true for $n-1$.  To simplify notation, set
$$
X=X_{p,n},\quad  G=\Aut_{X/K},\quadand H=\GG_a\rtimes U_n\rtimes \GG_m,
$$
and let $I\subset G$ be the inertia group scheme for the singularity $x_0\in X$.  Likewise, we set $X'= X_{p,n-1}$. By induction, $H'=\GG_a\rtimes U_{n-1}\rtimes \GG_m$ coincides with $G'=\Aut_{X'/K}$.  Also note that by the very definition of the group schemes, we have a canonical inclusion $H'\subset H$.  Moreover, with Proposition~\ref{inertia} we get an inclusion $H'\subset I$, and actually $H'=I\cap H$.  By Lemma~\ref{blowing-up} combined with \cite[Proposition~2.7]{Martin 2022}, the blowing-up morphism $f\colon X'\ra X$ is equivariant with respect to the action of $I$. We thus get inclusions $H'\subset I\subset G'=H'$ and infer $H'=I$.

The orbit map for the rational point $x_0\in X$ gives an inclusion $H/H' \subset G/I$ of closed subschemes inside~$X$.  This is actually contained in the scheme of singularities $Z=\Sing(X/K)$, according to \cite[Proposition~3.1]{Brion; Schroeer 2023}.  Our task is to show that the inclusion is an equality, and for this it suffices to verify that the coordinate rings have the same degree.  We already saw that $h^0(\O_{H/H'})=p^n$, and it remains to verify $h^0(\O_{G/I})\leq p^n$.  For this, we may assume that $K$ is algebraically closed.

According to \eqref{relation differentials} and Proposition~\ref{homogeneous ideal}, the scheme of singularities $Z=\Sing(X/K)$ has coordinate ring of the form $A= K[x_0,\ldots,x_{n-1},y]\,/(x_0^p,\ldots,x_{n-1}^p,y^{p-2})$.  In light of \cite[Section~III.3, Theorem~6.1]{Demazure; Gabriel 1970}, the homogeneous space $G/I$ has coordinate ring of the form $B=K[u_1,\ldots,u_r]\,/(u_1^{p^{\nu_1}},\ldots, u_r^{p^{\nu_r}})$ for some $r\geq 0$ and certain exponents $\nu_i\geq 1$. From the canonical surjection
$$
\varphi\colon A=\Gamma(\O_Z)\lra\Gamma(\O_{G/I})= B, 
$$ 
we infer $\nu_i=1$ and $r\leq n+1$.  Using the relation $y^{p-2}=0$, we see that $\varphi(y)\in \maxid_B$ must be contained in $\maxid_B^2$, and thus actually $r\leq n$.  It follows that $h^0(\O_{G/I})\leq p^n$, as desired.
\end{proof}

\section{Equivariant normality and twisting}
\label{Equivariant normality}

We now seek to construct twisted forms of our curves $X_{p,n}$ that are regular.  Our methods to achieve this apply in many other contexts, and we first give a general discussion about twisted forms, their regularity properties, and Brion's recent notion of equivariant normality.

Fix a ground field $K$, and let $X$ be a scheme.  Recall that another scheme $Y$ is called a \emph{twisted form} of $X$ if we have $Y\otimes L\simeq X\otimes L$ for some field extension $K\subset L$.  Such twisted forms may arise as follows: Suppose that a group scheme $G$ acts on $X$, and let $P$ be a $G$-torsor.  Then $G$ acts diagonally on $P\times X$, and the quotient
$$
{}^P\! X = G\backslash (P\times X)
$$
is a twisted form of $X$. Note that the diagonal action is free; hence the quotient exists as an algebraic space. Such quotients are not necessarily schematic (for concrete examples, see \cite{Schroeer 2022}).  However, if $G$ is finite and $X$ is covered by affine open sets that are $G$-stable, the twisted form is indeed a scheme (\textit{cf.} \cite[Section~III.2, Theorem~3.2]{Demazure; Gabriel 1970}).
 
\looseness=-1 Now suppose that we are in positive characteristic $p>0$.  It then may happen that a noetherian scheme with singularities has twisted forms where all singularities are gone.  We now describe a fairly general procedure to achieve this, relying on a combination of works of Brion and the second author \cite{Brion 2022b, Brion; Schroeer 2023,Schroeer 2007,Schroeer 2023c}.  For simplicity, we assume throughout that $X$ is a separated scheme of finite type that is geometrically integral.  It is \emph{normal} if all local rings $\O_{X,a}$ are integrally closed in the common function field $F=\Frac(\O_{X,a})$.  Equivalently, each finite modification $f\colon X'\ra X$ is an isomorphism.  Here the term \emph{modification} refers to an integral scheme $X'$, together with a proper surjective morphism $f\colon X'\ra X$ inducing a bijection on function fields.

In what follows, we suppose that $X$ is endowed with the action of a finite group scheme $G$.  We now consider only modifications $f\colon X'\ra X$ where $X'$ is a $G$-scheme and $f$ is equivariant. For brevity, we call such a datum an \emph{equivariant modification} or \emph{$G$-modification}. Examples are given by blowing-ups $X'=\Bl_Z(X)$ with respect to $G$-stable centers $Z\subset X$.  Note that for a given modification $X'$, there is at most one $G$-action on $X'$ making $f\colon X'\ra X$ equivariant.

One says that $X$ is \emph{equivariantly normal}, or \emph{$G$-normal}, if every finite equivariant modification $f\colon X'\ra X$ is an isomorphism.  This extremely useful notion was introduced and studied by Brion in \cite{Brion 2022b}. He showed that $X$ admits a finite equivariant modification $\tilde{X}$ that is equivariantly normal (\textit{cf.} \cite[Proposition 4.2]{Brion 2022b}).
It is actually unique up to unique equivariant isomorphism provided that $X$ is one-dimensional (\textit{cf.} \cite[Corollary 4.4]{Brion 2022b}).

From now on, we furthermore assume that our $G$-scheme $X$ is one-dimensional. One could also say that $X$ is a \emph{$G$-curve}.  Following \cite[Section~2]{Fanelli; Schroeer 2020} we write $\Sing(X/K)$ for the closed subscheme defined by the first Fitting ideal for $\Omega^1_{X/K}$. This is the set of points $a\in X$ where the local ring $\O_{X,a}$ fails to be geometrically regular, endowed with a canonical scheme structure.  Note that with respect to this scheme structure, it must be $G$-stable (\textit{cf.} \cite[Proposition~3.1]{Brion; Schroeer 2023}).  The existence of twisted forms that are regular is intimately related to equivariant normality.

\begin{theorem}
\label{sufficient regular twisting}
Let $P$ be a $G$-torsor. Then the twisted form $Y={}^P\!X$ is regular provided the following three conditions hold:
\begin{enumerate}
\item\label{srt-1} The curve $X$ is  $G$-normal.
\item\label{srt-2} The total space of the $G$-torsor  $P$ is reduced.
\item\label{srt-3} The reduction of the finite scheme $\Sing(X/K)$ is \'etale.
\end{enumerate}
\end{theorem}

\begin{proof}
The residue fields $\kappa(a)$ for the points $a\in\Sing(X/K)$ are separable, by assumption~\eqref{srt-3}, and so is their join $L$.  This ensures that the base-change $P_L$ remains reduced. Furthermore, the arguments for \cite[Proposition~4.10]{Brion 2022b} show that $X_L$ remains equivariantly normal. Replacing the ground field with $L$, we may assume that $\Sing(X/K)=\{a_1,\ldots,a_r\}$ comprises only rational points. Let $H_i\subset G$ be the inertia subgroup scheme and $Z_i=G\cdot a_i=G/H_i$ be the orbit for $a_i\in X$.  Clearly, the subscheme $Z_1\cup\ldots\cup Z_r$ and its complementary open set $U$ are $G$-stable.  The latter is geometrically regular, and so is the twisted form ${}^P\!U$.

It remains to verify that the integral curve $Y={}^P\!X$ is regular at the points $b\in{}^P\!Z_i$.  According to \cite[Theorem~4.13]{Brion 2022b}, the orbit $Z_i\subset X$ is an effective Cartier divisor.  In turn, its twist ${}^P\!Z_i\subset {}^P\!X$ is an effective Cartier divisor on ${}^P\!X$, so it suffices to verify that it is reduced. The latter becomes the quotient of $P\times G/H_i$ by the diagonal $G$-action, which can be identified with $H_i\backslash P$.  Its coordinate ring is a subring inside $\Gamma(P,\O_P)$, which can be seen as a ring of invariants, and $\Gamma(P,\O_P)$ is reduced by assumption.
\end{proof}

Note that condition~\eqref{srt-3} holds in particular if all $a\in\Sing(X/K)$ are rational points.  The first two conditions can be achieved after ground field extensions.

\begin{proposition}
\label{conditions after base-change}
Suppose that the curve $X$ is $G$-normal.  Then there is a field extension $K\subset L$ such that the following hold:
\begin{enumerate}
\item\label{cbc-1} The base-change $X_L$ is $G_L$-normal.
\item\label{cbc-2} There is a $G_L$-torsor $P$ whose total space is reduced.
\end{enumerate}
\end{proposition}
 
\begin{proof}
\eqref{cbc-2}~ Choose a geometrically integral quasi-projective $G$-scheme $U$ with generically free action.  This could arise from an embedding $G\subset H$ into a smooth group scheme of finite type or could arise from a projective scheme $X$ with $G=\Aut^0_{X/K}$, according to \cite[Proposition~1.7 or Theorem~2.1]{Brion; Schroeer 2023}.  The quotient $V=U/G$ is an integral quasi-projective scheme, and the quotient map $f\colon U\ra V$ induces a finite extension of the function field $L=k(V)$ by $E=k(U)$.  By construction, the reduced scheme $P=\Spec(E)$, viewed as an $L$-scheme, is a torsor with respect to the base-change $G_L=G\otimes_KL$.

\eqref{cbc-1}~ The above extension $K\subset E$ is separable because $U$ is geometrically reduced.  In turn, the subextension $L$ is also separable.  The arguments for \cite[Proposition~4.10]{Brion 2022b} show that $X_L$ remains equivariantly normal.
\end{proof}

In the reverse direction, we have the following result.  

\begin{theorem}
\label{sufficient equivariant normal}
Suppose that there exist a field extension $K\subset L$, a subgroup scheme $H\subset G_L$, and a $H$-torsor $P$ so that the twisted form ${}^P\!(X_L)$ is regular. Then the curve $X$ is $G$-normal.
\end{theorem}

\begin{proof}
According to \cite[Proposition~4.10 and Remark~4.3]{Brion 2022b}, it suffices to treat the case $L=K$ and $H=G$.  Let $X'={}^P\!X$ be the regular twisted form.  According to \cite[Lemma~3.1]{Schroeer; Tziolas 2023}, we have a canonical identification $\Aut_{X'/K}={}^P\!\Aut_{X/K}$ of the sheaves of automorphisms, where the term on the right is formed with respect to the conjugacy action of $G$ on $\Aut_{X/K}$.  Setting $G'={}^PG$, we get a homomorphism $G'\ra\Aut_{X'/K}$, hence a $G'$-action on $X'$.

Let $Z\subset X$ be a finite closed subscheme that is $G$-stable.  According to \cite[Theorem~4.13]{Brion 2022b}, we have to check that $Z$ is Cartier.  Its twist $Z'={}^P\!Z$ defines a finite closed subscheme inside $X'$. The latter is regular, so $Z'$ is Cartier.  Choose a point $p\in P$, and let $E=\kappa(p)$ be the resulting field extension.  The resulting trivialization of $P\otimes E$ defines an isomorphism $g\colon X\otimes E\ra X'\otimes E$ with $g(Z\otimes E)=Z'\otimes E$. It follows that $Z\otimes E$, and hence $Z$, is Cartier.
\end{proof}

Recall that a \emph{quasielliptic curve} is a regular curve $Y$ that is a twisted form of the \emph{rational cuspidal curve}
$$
X=\Spec K\left[T^2,T^3\right]\cup \Spec K\left[T^{-1}\right].
$$
Clearly, $\Sing(X/K)$ is a singleton, containing only the rational point $x_0$ given by $T^2=T^3=0$.  It turns out that quasielliptic curves exist only in characteristic two and three (compare with the discussion after Proposition~\ref{useful facts cohomology}).  For $p=2$, the rational vector field $D=T^{-2} \partial/\partial T$ satisfies $D^{[p]}=0$ and actually defines a global section $D\in H^0(X,\Theta_{X/K})$, hence corresponds to an action of $G=\alpha_p$, which is the Frobenius kernel of the additive group $\GG_a$.  The orbit $G\cdot x_0$ is the Cartier divisor defined by $T^2=0$.  For $p=3$, the same holds for $D=\partial/\partial T$ and the Cartier divisor $T^3=0$.

In both cases, we conclude that the rational cuspidal curve is equivariantly normal with respect to $G=\alpha_p$ (again by \cite[Theorem~4.13]{Brion 2022b}). For this group scheme, torsors with regular total space exist if and only if $K$ is imperfect (see for example \cite[Lemma~7.1]{Schroeer; Tziolas 2023}), and then quasielliptic curves exist by Theorem~\ref{sufficient regular twisting}.  Also note that $X$ is equivariantly normal with respect to any larger finite subgroup scheme inside the full automorphism group scheme (obvious, see \cite[Remark~4.3]{Brion 2022b}).  According to \cite[Proposition~6]{Bombieri; Mumford 1977}, we have an iterated semidirect product $\Aut_{X/K}=\GG_a\rtimes U\rtimes\GG_m$ for an infinitesimal group scheme $U$.  For $p=3$, it coincides with the copy of $\alpha_p$ described above, whereas for $p=2$, it has order $|U|=8$.

All this generalizes to our hierarchy of curves $X_{p,n}$.

\begin{theorem}
\label{equivariantly normal}
The curve $X=X_{p,n}$ is equivariantly normal with respect to the finite group scheme $U_n$ and locally of complete intersection.  Moreover, if there is a $U_n$-torsor $P$ so that the quotient $\bar{P}=U_{n-1}\backslash P$ is reduced, then the twisted form $Y={}^P\!X$ is regular.
\end{theorem}

\begin{proof}
The first assertion follows from \cite[Theorem~4.13 and Corollary~4.18]{Brion 2022b}.  If $P$ itself is reduced, the assertion on the twisted form $Y={}^P\!X$ directly follows from Theorem~\ref{sufficient regular twisting}.  Its proof actually shows our slightly stronger claim because the singularity $x_0\in X$ is rational, with orbit $U_n/U_{n-1}$.
\end{proof}

Note that after some ground field extension $K\subset L$, there is a $U_n$-torsor $P$ that is reduced, according to Proposition~\ref{conditions after base-change}, and our curve $X=X_{p,n}$ acquires twisted forms that are regular.  However, the construction of $L$ relies, via \cite[Proposition~1.7 and Theorem~2.1]{Brion; Schroeer 2023}, among other things on embeddings of $U_n$ into smooth group schemes $H$, and here we have little control over $\dim(H)$ and $\trdeg(L)$.

Also note that the above argument for locally complete intersection relying on equivariant normality is independent of the arguments relying on numerical semigroups in the proof of Proposition~\ref{conductor}.


\section{Non-abelian cohomology and semidirect products}
\label{Non-abelian cohomology}

In this section we review the general notions of \emph{torsors and twisting}, which will be used in the next section to understand the twisted forms of our curves $X_{p,n}$. The material is well known, but it is not easy to find suitable references that are general enough for our purposes, yet not burdened by over-abstraction.  Throughout, we are guided by \cite[Section~I.5]{Serre 1972} and \cite[Section~III.2]{Giraud 1971}.
 
Let $\shT=\Sh(\catC)$ be the topos of sheaves on some site $\catC$, having a final object $S$.  For any group-valued object $G\in\shT$, we write $H^0(S,G)$ for the group of global sections and $H^1(S,G)$ for the set of isomorphism classes of \emph{$G$-torsors} $P$. The latter is an object endowed with a $G$-action that is locally isomorphic to $P_0=G$ with the translation action.  Another widespread term is \emph{principal homogeneous spaces}.  For $G$ commutative, our $H^1(S,G)$ coincides with the sheaf cohomology groups.  In general, however, $H^1(S,G)$ is merely a set, containing the class of the trivial torsor $P_0=G$ as a distinguished element.

An object $\tilde{X}$ is called a \emph{twisted form} of an object $X$ if the two are locally isomorphic.  If $X$ has a $G$-action and $P$ is a $G$-torsor, we get such a twisted form by forming the quotient $\tilde{X} ={}^P\!X= P\wedge^G X = G\backslash (P\times X)$ with respect to the diagonal action $\sigma\cdot (p,x)=(\sigma p,\sigma x)$. Note that this could also be written as $(p\sigma^{-1},\sigma x)$.  For $G=\Aut_{X/S}$, the above construction gives an identification between the non-abelian cohomology $H^1(S,G)$ and the set $\operatorname{Twist}(X)$ of isomorphism classes of twisted forms $\tilde{X}$ of the object $X$.  In any case, the $G$-action on $X$ induces a conjugacy action on $\Aut_{X/S}$, and we have $\Aut_{(P\wedge^GX)/S}=P\wedge^G\Aut_{X/S}$; compare for example with \cite[Lemma~3.1]{Schroeer; Tziolas 2023}.

Now suppose $X=G$ as sheaves of sets without group laws, and consider the homomorphism $G\times G^\op\ra\Aut_{X/S}$ given by $(\sigma_1,\sigma_2)\cdot x=\sigma_1 x\sigma_2^{-1}$.  One easily checks that the map is equivariant with respect to factorwise conjugation $\eta\cdot (\sigma_1,\sigma_2)=(\eta\sigma_1\eta^{-1},\eta\sigma_2\eta^{-1})$ and conjugation with inner automorphisms on $\Aut_{X/S}$. In turn, we get an induced homomorphism
$$
{}^P\!G\times{}^P\!G^\op\lra\Aut_{(P\wedge^GX)/S}=\Aut_{P/S}.
$$
Note that the equation stems from the identification $P\wedge^GX=P$.  The above endows each $G$-torsor $P$ with the \emph{additional structure} of a ${}^P\!G$-torsor and a ${}^P\!G^\op$-torsor. Furthermore, ${}^P\!G^\op$ is the automorphism group object of $P$ as a $G$-torsor, and $G$ is the automorphism group object of $P$ as a ${}^P\!G^\op$-torsor.  In turn, we get what we like to call the \emph{torsor translation map}
\begin{equation}
\label{torsor translation}
H^1\left(S,{}^P\!G\right)\lra H^1(S,G),\quad T\longmapsto P\wedge^{{}^P\!G}T,
\end{equation}
where the quotient on the right is formed with respect to the action $\tilde{\sigma}\cdot (p,t)=(p\tilde{\sigma}^{-1},\tilde{\sigma} t)$ and the $G$-action on $P\wedge^GT$ stems from the action on the first factor $P$.  The map \eqref{torsor translation} is bijective but does not respect the distinguished points: rather, it sends $T_0={}^P\!G$ to $P=P\wedge^{{}^P\!G}T_0$.
 
Now suppose that we have a short exact sequence
\begin{equation}
\label{extension}
1\lra A\lra B\stackrel{\pr}{\lra} C\lra 1
\end{equation}
of group objects. Then the group $H^0(S,C)$ acts from the right on the set $H^1(S,A)$ in the following way: For each global section $c\in H^0(S,C)$, the fiber $B_c=\pr^{-1}(c)$ with respect to the surjection $B\ra C$ carries compatible $A$-torsor structures from both sides, coming from the group law in $B$.  We now define $c\cdot [P] = [B_c\wedge^A P]$.  The stabilizer group at each torsor class is the subgroup of global sections $c\in H^0(S,C)$ where the set of global sections $H^0(S,B_c)$ is non-empty.

Let us write $H^1(S,A)/H^0(S,C)^\op = H^0(S,C)\backslash H^1(S,A)$ for the quotient of the action.  Using the distinguished point in $H^1(S,A)$, the orbit map $c\mapsto B_c$ yields $H^0(S,C)\ra H^1(S,A)$. The latter serves as a \emph{connecting map} and yields a \emph{six-term sequence} of sets
\begin{equation*}
1\lra H^0(S,A)\lra H^0(S,B)\lra H^0(S,C)\stackrel{\partial}{\lra} H^1(S,A)\lra H^1(S,B)\lra H^1(S,C).
\end{equation*}
The maps on the right come from extension of structure groups.  Here all arrows preserve the distinguished points, and in degree zero the above is an exact sequence of groups.

The group object $B$ acts on itself and its quotient $C=B/A$ via conjugacy.  On the normal subgroup $A$, we have an induced action.  Twisting with respect to the $B$-actions gives another exact sequence
\begin{equation}
\label{twisted extension}
1\lra {}^P\!A\lra {}^P\!B\lra {}^P\!C\lra 1,
\end{equation}
which also yields a six-term sequence.  Note that in general there is no map relating $H^1(S,A)$ and $H^1(S,{}^P\!A)$ because the $B$-action on $A$ usually fails to be inner.

We now choose for each $C$-torsor $\bar{P}$ whose class belongs to the image of the mapping $H^1(S,B)\ra H^1(S,C)$ some $B$-torsor $P$ with $C\wedge^B T\simeq \bar{P}$.  As in \cite[Section~5.5, Corollary 2]{Serre 1972}, one has the following.

\begin{theorem}
\label{torsor decomposition}
The first cohomology of $B$ can be written as a disjoint union
$$
H^1(S,B) = \bigcup H^1\left(S,{}^{P}\!A\right)/H^0\left(S,{}^{\bar{P}}\!C\right)^\op 
$$
running over all $[\bar{P}]$ from the image of $H^1(S,B)\ra H^1(S,C)$.  The inclusions are obtained by composing the induced maps $H^1(S,{}^{P}\!A)\ra H^1(S,{}^{P}\!B)$ with the torsor translation maps $ H^1(S, {}^{P}\!B) \ra H^1(S,B)$ given by \eqref{torsor translation}.
\end{theorem}

We are interested in cases where the above simplifies. Recall that $\pr\colon B\ra C$ is the canonical epimorphism. Let us call a morphism $s\colon C\ra B$ a \emph{set-theoretical section} if $\pr\circ s=\id_C$. The point here is that $s$ does not have to preserve the group laws.  The resulting $A\times C\ra B$ given by $(a,c)\mapsto a\cdot s(c)$ is an isomorphism of objects that does not necessarily respect the group laws.  The latter is determined by the two-cocycle $\tau\colon C^2\ra A$ defined by $s(cc')=\tau_{c,c'}\cdot s(c) s(c')$.  We like to indicate this situation by writing
$$
B=A\tilde{\times} C=A\tilde{\times}_\tau C
$$
and say that the extension $C$ is \emph{set-theoretically split}. Note that this always holds in the category of groups but often fails in the category of group schemes (compare with \cite[around Theorem~8.5]{Schroeer 2023b}). Also note that this property is not necessarily preserved in twisted extensions \eqref{twisted extension}.  If $s$ respects the group laws, the above becomes a semidirect product $B=A\rtimes C=A\rtimes_\varphi C$, where $\varphi$ is given by conjugacy, and the extension is called \emph{split}.

\begin{corollary}
\label{simplified torsor decomposition}
Suppose that for all $\bar{P}$ as above, the extension \eqref{twisted extension} is set-theoretically split or the group $H^0(S,{}^{\bar{P}}\!C)$ is trivial.  Then we have a disjoint union
$$
H^1(S,B) = \bigcup   H^1\left(S,{}^{P}\!A\right),
$$
running over all $[\bar{P}]$ from the image of $H^1(S,B)\ra H^1(S,C)$.  The latter map is actually surjective provided that the extension \eqref{extension} is split.
\end{corollary}

\begin{proof}
Set $C'={}^{\bar{P}}\!C$ and $A'={}^{P}\!A$.  If $B'= {}^{P}\!B$ is set-theoretically split, all fibers over $c'\in H^0(S,C')$ are trivial torsors, and hence the action of $H^0(S,C')$ on $H^1(S,A')$ is trivial. The same holds, of course, if the group $H^0(S,C')$ itself is trivial.  So the theorem implies the first assertion.

If the projection $\pr\colon B\ra C$ admits a section $s$ that respects the group structure, the induced map on cohomology is right inverse to $H^1(S,B)\ra H^1(S,C)$, so the latter is surjective.
\end{proof}

In particular, for $B=A\rtimes C$ satisfying the assumptions of the corollary, we get a disjoint union
$$
H^1(S,B)=\bigcup H^1\left(S,{}^{P}\!A\right)
$$
running over all $[\bar{P}]\in H^1(S,C)$.  It is convenient to regard its elements as ``pairs'' $(\gamma,\alpha)$ with $\gamma=[P]\in H^1(S,C)$ and $\alpha\in H^1(S,{}^{P}\!A)$. If $H^1(S,\Aut_{A/S})$ is a singleton, the choice of isomorphisms $h\colon {}^{P}\!A\ra A$ indeed identifies the above with the product $H^1(S,C)\times H^1(S,A)$, independently of the $h$.  Note that this carries a canonical group structure if $A$ and $C$ are commutative, which may happen without $B$ being commutative.

\section{Description of the set of   twisted forms}
\label{Set torsors}

Let $K$ be a ground field of characteristic $p>0$, and set $S=\Spec(K)$. Using the general results of the previous section, we seek to compute the first non-abelian cohomology for the iterated semidirect products $\GG_a\rtimes U_n\rtimes \GG_m=\Aut_{X_{p,n}/K}$ and thereby the set of isomorphism classes of twisted forms for $X_{p,n}$.  We are able to do so for $1\leq n\leq 2$.

Throughout, we work over the site $\catC=(\Aff/S)$, endowed with the fppf topology.  Let us start with some general useful facts.

\begin{proposition}
\label{useful facts cohomology}
The following hold for group schemes $G$ of finite type:
\begin{enumerate}
\item\label{ufc-1}
  If\, $G=N\,\widetilde{\times}\,\GG_m$ for some group scheme $N$ of finite type, then the canonical map $H^1(S,N)\ra H^1(S,G)$ is bijective.
\item\label{ufc-2}
  In the special case $G=\GG_a^{\oplus r}\,\widetilde{\times}\,\GG_m$, the set $H^1(S,G)$ is a singleton.
\item\label{ufc-3}
  If\, $G$ is infinitesimal,  the group  $H^0(S,G)$ is trivial.
\end{enumerate}
\end{proposition}

\begin{proof}
We have $H^1(S,\GG_m)=0$ by Hilbert 90, so~\eqref{ufc-1} follows from Corollary~\ref{simplified torsor decomposition}.  With $N=\GG_a^{\oplus r}$ we use \cite[Theorem~III.3.7]{Milne 1980} and Serre's vanishing theorem to get $H^1(S,N)=0$ and~\eqref{ufc-2} follows as well.  Assertion~\eqref{ufc-3} is obvious because an infinitesimal group scheme has $G_\red=S$.
\end{proof}

Note that for $p\geq 5$, the automorphism group scheme for the rational cuspidal curve $X=\Spec K[T^2,T^3]\cup\Spec K[T^{-1}]$ is given by $G=\GG_a\rtimes\GG_m$, so this curve has no twisted forms besides itself. This purely cohomological argument shows again that quasielliptic curves are confined to characteristic $p\leq 3$.  The following observation will also be useful.

\begin{lemma}
\label{isomorphic twisted forms}
Let $G_1$ and $G_2$ be twisted forms of\, $\GG_a$. If they are isomorphic as schemes, they are isomorphic as group schemes.
\end{lemma}

\begin{proof}
Let $f\colon G_1\ra G_2$ be an isomorphism of schemes. Composing with a translation, we may assume $f(e_1)=e_2$.  To verify that $f$ respects the group law, we may assume $K=K^\alg$, and this reduces to the case $G_1=G_2=\GG_a$.  The induced map $\varphi\colon K[T]\ra K[T]$ on the coordinate ring is given by $\varphi(T)=\lambda T + \mu$ for some $\lambda,\mu\in K$.  We have $\lambda\neq 0$ because $f$ is non-constant, and $\mu=0$ because $f$ respects the origin. So for each $\alpha\in\GG_a(R)$, we have $f(\alpha)=\lambda\alpha$, which respects the group laws.
\end{proof}

For each pair $\Phi,\Psi\in K[u]$ of additive polynomials with $\gcd(\Phi,\Psi)\neq 0$ (in other words, not both polynomials vanish), the resulting homomorphism $(\Phi,-\Psi)\colon \GG_a^{\oplus 2} \ra\GG_a$ given by $(u,v)\mapsto \Phi(u)-\Psi(v)$ is an epimorphism. The short exact sequence
\begin{equation}
\label{extension unipotent}
0\lra U_{\Phi,\Psi}\lra \GG_a^{\oplus 2}\xrightarrow{(\Phi,-\Psi)}\GG_a\lra 0
\end{equation}
defines a unipotent group scheme $U_{\Phi,\Psi}$, and the resulting long exact sequence yields
\begin{equation}
\label{cohomology phi psi}
H^1(S,U_{\Phi,\Psi})= K/\left\{\Phi(u)-\Psi(v)\mid u,v\in K\right\}.
\end{equation}
By Russell's theorem \cite[Theorem~2.1]{Russell 1970}, every twisted form of the additive group is isomorphic to $U_{\Phi,\Psi}$ with $\Phi(u)=u^{p^n}$ for some $n\geq 0$, and $\Psi(u)$ separable. The following sheds further light on this.

\begin{proposition}
\label{twist additive group}
The unipotent group scheme $U_{\Phi,\Psi}$ is a twisted form of\, $\GG_a$ if and only if $\gcd(\Phi,\Psi)=u$ inside the euclidean domain $K[u]$.
\end{proposition}

\begin{proof}
  It suffices to treat the case that $K$ is algebraically closed. Set $U=U_{\Phi,\Psi}$.  First suppose that there are $a,b\in k[u]$ with $a\Phi-b\Psi=u$.  These yield a section for $(\Phi,-\Psi)\colon \GG_a^{\oplus 2}\ra\GG_a$, defined via $u\mapsto (a(u),b(u))$, which does not have to preserve the group laws.  It induces an identification $\GG_a^{\oplus 2}=U\times\GG_a$ of schemes.  In turn, the coordinate ring $A=\Gamma(U,\O_U)$ has the property $A[y]=K[x,y]$.  According to Zariski cancellation (see \cite[Corollary~2.8]{Abhyankar; Heinzer; Eakin 1972}), the underlying scheme is isomorphic to the affine line $\AA^1$.  By Lazard's theorem (see \cite[Section~IV.4, Theorem~4.1]{Demazure; Gabriel 1970}), we must have $U\simeq \GG_a$ as group schemes.

Conversely, suppose $\gcd(\Phi,\Psi)\neq u$.  Since $K$ is algebraically closed, we have $\Phi(u)=\prod_{\omega\in A}(u-\omega)^{p^m}$ and $\Psi(u) = \prod_{\omega\in B}(u-\omega)^{p^n}$ for some exponents $m,n\geq 0$ and some finite subgroups $A,B\subset K$. Their intersection is non-zero, by the assumption on the gcd.  Consequently, we can write $\Phi(u)= \Phi_1(h(u))$ and $\Psi(u)=\Psi_1(h(u))$ for some additive polynomial of the form $h(u)=\prod_{i=0}^{p-1}(u-i\omega_0)$, with some non-zero $\omega_0\in K$.  In turn, the projection $(\Phi,-\Psi)\colon \GG_a^{\oplus 2}\ra\GG_a$ factors over the morphism $h\colon \GG_a\ra\GG_a$.  So the kernel $U_{\Phi,\Psi}$ is disconnected or non-reduced.
\end{proof}
 
\begin{proposition}
\label{cohomology iterated semidirect product}
For each group scheme of the form $G=\GG_a\rtimes Q\rtimes\GG_m$, where $Q$ is any infinitesimal group scheme of finite type, we have a canonical identification
$$
H^1(S,G) = \bigcup_{\alpha\in H^1(S,Q)}  K/\{\Phi_\alpha(u)-\Psi_\alpha(v)\mid u,v\in K\}
$$
for certain   additive polynomials $\Phi_\alpha,\Psi_\alpha\in K[u]$ with $\gcd(\Phi_\alpha,\Psi_\alpha)=u$.
\end{proposition}

\begin{proof}
By Proposition~\ref{useful facts cohomology}, the canonical map $H^1(S,\GG_a\rtimes Q)\ra H^1(S,G)$ is bijective.  Since $Q$ is infinitesimal, we can apply Proposition~\ref{simplified torsor decomposition}, and the assertion follows with \eqref{cohomology phi psi}.
\end{proof}

Roughly speaking, to understand this cohomology of $G$, one has to understand the cohomology of $Q$ and the dependence of the additive polynomials $\Phi_\alpha,\Psi_\alpha$ on the class $\alpha$. We now seek to unravel this with $G=\GG_a\rtimes U_n\rtimes\GG_m$.  For $n=1$, the term in the middle becomes $U_1=\alpha_p$.  The short exact sequence $0\ra\alpha_{p^n}\ra\GG_a\stackrel{F^n}{\ra}\GG_a\ra 0$ yields an identification $H^1(S,\alpha_{p^n})=K/K^{p^n}$ for every $n\geq 1$.  The dependence on the additive polynomials can be described as follows.

\begin{proposition}
\label{dependence n=1}
For the $\alpha_p$-torsor $P=\Spec K[y]\,/(y^p-\alpha)$, the additive polynomials $\Phi(u)=u^p$ and $\Psi(v)=v-\alpha v^p$ give the twisted form ${}^P\!\GG_a=U_{\Phi,\Psi}$.
\end{proposition}

\begin{proof}
Set $B=K[x,y]\,/(y^p-\alpha)$.  By definition, the coordinate ring for ${}^P\!\GG_a$ is the subring $A\subset B$ of elements that are invariant under $x\mapsto x+\lambda x^p$ and $y\mapsto y+\lambda$, for all group elements $\lambda\in \alpha_p(R)$ and all rings~$R$.  Clearly, $v=x^p$ and $u=x-x^py$ are such invariants, satisfying the relation $u^p=v-\alpha v^p$.  Its partial derivatives generate the unit ideal, and we conclude that the subring $A'\subset A$ generated by $u$ and $v$ is regular and one-dimensional and can be identified with the residue class ring $K[u,v]\,/(u^p-v+\alpha v^p)$.  The composite extension $K(x^p)\subset\Frac(A')\subset\Frac(A)\subset\Frac(B)$ has degree~$p^2$, and the outer steps have degree~$p$.  Consequently, $A'=A$. The assertion now follows from Lemma~\ref{isomorphic twisted forms}.
\end{proof}

To tackle the case $n=2$, we use the   short exact sequence
$$
0\lra\alpha_p\lra U_2\lra \alpha_{p^2}\lra 0,
$$
where the inclusion on the left is $\lambda\mapsto 1+\lambda x^{p^2}$ and the surjection on the right $(1+\lambda_1x^p+\lambda_2x^{p^2})\mapsto \lambda_1$.  Given $\alpha,\beta \in K$, the finite scheme $P=P_{\alpha,\beta}$ defined by
\begin{equation}
\label{non-commutative torsor}
P(R)=\left\{(y,z)\in R^2\mid \text{$y^{p^2}=\alpha$ and $z^p=\beta +\alpha y^p$}\right\}
\end{equation}
carries a $U_2$-action via the formula $(\lambda_1,\lambda_2)\cdot (y,z) = ( \lambda_1+y, \lambda_2+z +\lambda_1 y^p)$.  One easily checks that this indeed takes values in $P(R)$, that it satisfies the axioms for group actions, and that the action is free and transitive. The induced $\alpha_{p^2}$-torsor $\bar{P}$ is obtained from $P$ as a quotient by $\alpha_p$, in other words, by the action of $\lambda_2$. This yields $\bar{P}(R)=\{y\in R\mid y^{p^2}=\alpha\} $.  In turn, we get the description $H^1(S,U_2) = \bigcup_{K/K^{p^2}} K/K^p$.  It remains to express the twisted form ${}^P\!\GG_a$ in terms of additive polynomials.

\begin{proposition}
\label{dependence n=2}
For the $U_2$-torsor $P$ as above, the additive polynomials
$$
\Phi(u)=u^{p^2}\quadand \Psi(u)= u - \alpha u^p - \beta^p u^{p^2}
$$
give the twisted form  ${}^P\!\GG_a=U_{\Phi,\Psi}$.
\end{proposition}

\begin{proof}
Set $B=K[x,y,z]\,/(y^{p^2}-\alpha,z^p-\beta-\alpha y^p)$.  The coordinate ring for ${}^P\!\GG_a$ is the subring $A\subset B$ of elements that are invariant under
$$
x\longmapsto x+\lambda_1 x^p+\lambda_2x^{p^2},\quad   y\longmapsto y+\lambda_1,\quadand z\longmapsto z + \lambda_2 + \lambda_1y^p
$$
for all group elements $(\lambda_1,\lambda_2)\in U_2(R)$ and all rings $R$.  One easily checks that $v=x^{p^2}$ and $u=x-x^py-x^{p^2}z +x^{p^2} y^{p+1}$ are invariant and that these invariants satisfy the relation $u^{p^2} = v - \alpha v^p - \beta^p v^{p^2}$.  The argument concludes as in the preceding proof.
\end{proof}

Note that the invariant $u$ can be found by starting with the non-invariant $x$ and successively adding monomials to cancel non-invariance.  Collecting all the above, we have determined the non-abelian cohomology  for $G_n=\GG_a\rtimes U_n\rtimes\GG_m$ in the cases $1\leq n\leq 2$. 

\begin{theorem}
\label{cohomology set}
With the above notation, we have 
$$
H^1(S,G_1) = \bigcup_{\alpha}  K/\left\{u^p-v+\alpha v^p\mid u,v\in K\right\},
$$
where the union runs over all $\alpha\in K/K^p$, and 
$$
H^1(S,G_2)= \bigcup_{(\alpha,\beta)} K/\left\{u^{p^2}- v - \alpha v^p - \beta^p v^{p^2}\mid u,v\in K\right\},
$$
where the union runs over $(\alpha,\beta)\in \bigcup_{K/K^{p^2}} K/K^p$ with $\alpha\in K/K^{p^2}$ and $\beta\in K/K^p$.
\end{theorem}

Note that for $3\leq p^n\leq 4$, this gives back, in an intrinsic fashion, Queen's descriptions for quasielliptic curves (\textit{cf.} \cite{Queen 1971, Queen 1972}).

Also note that for $n=1$, the group $K/\{u^p-v+\alpha v^p\mid u,v\in K\}$ is trivial, provided that $K$ is separably closed or $\alpha\in K^p$.  It follows that the Frobenius map
$$
H^1(S,G_1)\lra H^1(S,G_1),\quad P\longmapsto P^{(p)}=P\otimes{}_FK
$$
is trivial, in the sense that it sends every class to the distinguished class. In particular, every twisted form $Y$ of $X_{p,1}$ is untwisted by Frobenius pullback and becomes a rational curve (compare with \cite{Hilario; Stoehr 2022}).  Likewise, for $n=2$, the group $K/\{u^{p^2}- v - \alpha v^p - \beta^p v^{p^2}\mid u,v\in K\}$ is trivial if $K$ is separably closed or $\alpha\in K^{p^2}$, $\beta\in K^p$.  Now the map $P\mapsto P^{(p^2)}$ is trivial, and every twisted form $Y$ of $X_{p,2}$ gets untwisted by the second Frobenius pullback.

With the notation from the theorem, let $T$ be a $G_n$-torsor and $\alpha\in K/K^{p^n}$ be the ensuing class, for $1\leq n\leq 2$. Write $\tilde{X}$ for the twisted form of $X=X_{p,n}$ corresponding to $T$.

\begin{proposition}
\label{twisted form regular}
With the above notation, the curve $\tilde{X}$ is regular provided that $\alpha\in K/K^{p^n}$ does not belong to $K^p/K^{p^n}$.
\end{proposition}

\begin{proof} 
The $G_n$-torsor $T$ is induced from some torsor $P$ with respect to $\GG_a\rtimes U_n$, according to Proposition~\ref{useful facts cohomology}.  By construction, the class $\alpha\in K/K^{p^n}$ corresponds to the quotient $\bar{P}= (\GG_a\rtimes U_{n-1})\backslash P$, and the latter has coordinate ring $K[T]\,/(T^{p^n}-\alpha)$, where we also write $\alpha$ for the scalar rather than the class. The coordinate ring is reduced, in light of our assumption.  According to Theorem~\ref{equivariantly normal}, the curve $\tilde{X}$ is regular.
\end{proof}  

It should be possible to extend the above results to all $n\geq 1$. For this, one has to find an inductive description for the $U_n$-torsors, analogous to \eqref{non-commutative torsor}. The main problem is to cope with the non-commutativity involved in the torsors.


\end{document}